\documentclass{amsart}
\usepackage[all]{xy}
\usepackage{amsmath}
\usepackage{amssymb}
\usepackage{amsthm}
\usepackage{amscd}
\usepackage{color}

\newcommand{\M}{\mathcal{M}}

\newcommand{\MM}{\mathfrak{M}}

\newcommand{\LL}{\mathcal{L}}

\newcommand{\U}{\mathcal{U}}

\newcommand{\V}{\mathcal{V}}
\newcommand{\OO}{\hat{O}}

\newcommand{\Z}{\Bbb Z}

\newcommand{\RR}{\Bbb R}
\newcommand{\TT}{\Bbb T}

\newcommand{\QQ}{\Bbb Q}
\newcommand{\C}{\Bbb C}

\newcommand{\aveN}{\frac{1}{N}\sum_{n\leq N} }
\newcommand{\aveM}{\frac{1}{M}\sum_{n \leq M} }
\newcommand{\limaveN}{\lim_{N\to\infty} \aveN}

\renewcommand{\hat}[1]{\widehat{#1}}
\newcommand{\ol}[1]{\overline{#1}}

\newcommand{\veps}{\varepsilon}

\newcommand{\mps}{measure-preserving system}

\newtheorem{ex}{Example}

\newtheorem{theorem}{Theorem}

\newtheorem{definition}[theorem]{Definition}

\newtheorem{proposition}[theorem]{Proposition}
\newtheorem{prop}[theorem]{Proposition}

\newtheorem{lemma}[theorem]{Lemma}

\newtheorem{remark}[theorem]{Remark}

\DeclareMathOperator{\lcm}{lcm}

\DeclareMathOperator{\N}{\mathbb{N}}

\numberwithin{theorem}{section}

\begin{document}
\title{(Uniform) Convergence of Twisted Ergodic Averages}
\author{Tanja Eisner}
\address{ Institute of Mathematics, University of Leipzig\\
P.O. Box 100 920, 04009 Leipzig, Germany}
\email{eisner@math.uni-leipzig.de}

\author{Ben Krause}
\address{UCLA Math Sciences Building\\
         Los Angeles,CA 90095-1555}
\email{benkrause23@math.ucla.edu}
\date{\today}
\maketitle

\begin{abstract}
Let $T$ be an ergodic measure-preserving transformation on a non-atomic probability space $(X,\Sigma,\mu)$.
%
We prove uniform extensions of the Wiener-Wintner theorem in two settings:

For averages involving weights coming from Hardy field functions, $p$:
\[ \left\{ \frac{1}{N} \sum_{n\leq N} e( p(n) ) T^{n}f(x) \right\} \]
and for ``twisted'' polynomial ergodic averages:
\[ \left\{ \frac{1}{N} \sum_{n\leq N} e(n \theta) T^{P(n)}f(x) \right\} \]
for certain classes of badly approximable $\theta \in [0,1]$.

We also give an elementary proof that the above twisted polynomial averages converge pointwise $\mu$-a.e.\ for $f \in L^p(X), \ p >1,$ and arbitrary $\theta \in [0,1]$.
\end{abstract}

\section{Introduction}

Let $T$ be an ergodic invertible measure-preserving transformation on  a non-atomic probability space $(X,\Sigma,\mu)$, and denote by $T$ its Koopman operator given by
\[
(Tf)(x):= f( T x).
\]
The study of pointwise convergence of averages formed from the iterates $\{T^n\}$ began in $1931$ with the classical pointwise ergodic theorem of Birkhoff \cite{BI}:
\begin{theorem}
For any $f \in L^1(X)$ the averages
\[ \left\{ \frac{1}{N} \sum_{n \leq N} T^nf(x) \right\} \]
converge $\mu$-a.e. to the space mean $\int_X f\,d\mu$.
\end{theorem}

This result was extended to more general polynomial averages in the late eighties by Bourgain, who proved the following celebrated theorem \cite{B1}:

\begin{theorem}
Let $P$ be a polynomial with integer coefficients. For any $f \in L^p(X)$, $p >1$,
the averages
\[ \left\{ \frac{1}{N} \sum_{n \leq N} T^{P(n)} f(x) \right\} \]
converge $\mu$-a.e..
\end{theorem}

Prior to Bourgain's polynomial ergodic theorem, another generalization of Birkhoff's theorem was announced by Wiener and Wintner \cite{WW}, see Assani \cite{AWW} for more information and various proofs (here and throughout the paper, $e(t):= e^{2\pi i t}$ denotes the exponential):
\begin{theorem}[Wiener-Wintner]
For every $f \in L^1(X)$ there exists a subset $X' \subset X$ of full measure so that
the weighted averages
\[ \left\{ \frac{1}{N} \sum_{n \leq N} e(n \theta) T^n f(x) \right\} \]
converge for all $x \in X'$ and every $\theta \in [0,1]$ .
\end{theorem}
\begin{remark}
Bourgain \cite{BWW} observed that for $f$ orthogonal to eigenfunctions (i.e., to the Kronecker factor) of $T$ the above averages converge \emph{uniformly} in $\theta$ to zero. In addition, if $(X,T)$ is uniquely ergodic and $f$ is
continuous, then convergence is uniform in $x\in X$ too, see Assani \cite{AWW}.
\end{remark}

This result was in turn generalized to polynomial weights by Lesigne in \cite{lesigneWW}:

\begin{theorem}[Wiener-Wintner theorem for polynomial weights]
For every $f \in L^1(X)$ there exists a subset $X' \subset X$ of full measure so that
the weighted  averages
\[ \left\{ \frac{1}{N} \sum_{n \leq N} e(P(n)) T^n f(x) \right\} \]
converge for all $x \in X'$ and all real polynomials $P\in \RR[\cdot]$.
\end{theorem}
Although this list is by no means comprehensive, we remark that Lesigne's result has been further generalized and ``uniformized'' by Frantzikinakis \cite{frantWW}, Host, Kra \cite{HK}, and Eisner, Zorin-Kranich \cite{EZ}. 
We also mention here that the corresponding characteristic factor for Lesigne's averages, the \emph{Abramov factor}, is induced by generalized eigenfunctions of $T$ and is larger than the Kronecker factor; in particular, the limit of Lesigne's averages can be non-zero for functions which are orthogonal to the Kronecker factor.

\smallskip

There are two main aims of this paper.

We first extend Lesigne's result to weights coming from \emph{Hardy field} functions (see \S 2 below for the precise definition and references). 
As the weights come from Hardy field functions which are "far" from polynomials (defined below) the corresponding weighted ergodic averages always converge to zero -- unlike in the case of polynomial weights.
More precisely, our first main result in this direction can be stated as follows. (We refer the reader to $\S 2$ below for the precise definition of the class $\M_{\delta,M,m}$; informally, these are smooth functions which are uniformly ``$(\delta,M,m)$-far'' from the class of polynomials.)

\begin{theorem}\label{hardy}
For $0 < \delta < 1/2$, $m \geq 0$, and $M \geq 1$, consider the class of functions $\M_{\delta,M,m}$. Then for every $f \in L^1(X)$ there exists a subset $X' \subset X$ of full measure so that the averages
\begin{equation}\label{eq:hardy} 
\left\{ \frac{1}{N} \sum_{n \leq N} e(p(n)) T^n f(x) \right\} 
\end{equation}
converge to zero \emph{uniformly} in $p \in \M_{\delta,M,m}$ for all $x \in X'$, i.e.,
\[ \sup_{p \in \M_{\delta,M,m}} \left| \frac{1}{N} \sum_{n\leq N} e(p(n))T^nf(x) \right| \to 0 \quad \mu-\text{a.e.}.\]
Moreover, if $(X,\mu,T)$ is uniquely ergodic and $f \in C(X)$, one has moreover
\[ \sup_{p \in \M_{\delta,M,m}} \left\| \frac{1}{N} \sum_{n\leq N} e(p(n))T^nf  \right\|_{\infty} \to 0. \]
\end{theorem}
\begin{remark}
The requirement of invertibility is not used (or needed) in the proof of this theorem.
\end{remark}

We show moreover that for $p\in\U$ with type less than $m\in\N$, the averages (\ref{eq:hardy}) are bounded by the $(m+1)$st Gowers-Host-Kra uniformity seminorm of $f$ with certain quantitative uniformity in $p$, see Theorem \ref{thm:Hardy-seminorms} below for the precise formulation. This implies that for functions $f$ which are orthogonal to some Host-Kra factor, the averages (\ref{eq:hardy}) converge to $0$ uniformly in $p\in \U$ with corresponding growth bound.

\smallskip

Our second  aim is to study weighted averages with polynomial powers,
i.e., a combination of the Wiener-Wintner and Bourgain averages. Specifically, we consider pointwise convergence of weighted averages of the form
\[ \left\{ \frac{1}{N} \sum_{n \leq N} e(n \theta) T^{P(n)}f(x) \right\} \]
where $\theta \in [0,1]$ and $P$ is a polynomial with integer coefficients.

We state our main result in this direction in the case of the squares.
\begin{theorem}\label{badapp}
Let $0 < c <1$ and $E = E_c \subset [0,1]$ be a set of $c$-badly approximable numbers with upper Minkowski dimension strictly less than $1/16$.
Then for every $f \in L^2(X)$ there exists a subset $X' \subset X$ of full measure so that
\[\lim_{N\to\infty} \sup_{\theta \in E_c } \left| \frac{1}{N} \sum_{n\leq N} e(n\theta)T^{n^2} f(x) \right| = 0\]
holds for every $x\in X'$.
\end{theorem}
We recall that a number, $\theta$, is said to be \emph{$c$-badly approximable} if for all $\frac{p}{q}$ reduced fractions, we have the lower bound
\[ \left| \theta - \frac{p}{q} \right| \geq \frac{c}{q^2}.\]
(Note that by Dirichlet's principle we automatically have $0< c< 1$.)

A natural follow-up question concerns the behavior of the twisted square means in the case where $\theta$ is \emph{not} badly approximable. We provide a new (see the remark below), elementary proof of the following
\begin{theorem}\label{conv}
Let $P$ be as above, and $\theta \in [0,1]$ be arbitrary. For any $f \in L^p(X)$, $p >1$, there exists a subset $X' \subset X$ of full measure so that the averages
\[ \left\{ \frac{1}{N} \sum_{n \leq N} e(n\theta) T^{P(n)} f(x) \right\} \]
converge pointwise on $X'$.
\end{theorem}
\begin{remark}
By Bourgain's $L^p$-maximal inequality for polynomial ergodic averages \cite[\S 7]{B1}, the set of functions for which convergence holds is closed in $L^p$ for each $p>1$. Since Bourgain's \cite[Theorem 6]{B0} establishes the theorem for $L^2$-functions, this result is already known (this argument was also recently recovered as a corollary of Mirek and Trojan \cite{MT}).
 However, Bourgain's (and Mirek and Trojan's) argument is based off a study of $\Z^2$-actions, and transference to $\RR^2$; the novelty of our approach is that we conduct our analysis entirely in the one-dimensional setting.
\end{remark}

Note that related results concerning Wiener-Wintner type convergence for linear and polynomial weights for the double ergodic theorem were announced in Assani, Duncan, Moore \cite{ADM} and Assani, Moore \cite{AM}. Moreover, Theorem \ref{conv} gives a partial answer to the general question of the polynomial return time convergence, i.e., whether for every (ergodic) invertible system $(X,\mu, T)$, $f\in L^\infty(X,\mu)$ and a polynomial $P$ with integer coefficients, the sequence $\{f(T^{p(n)}x)\}$ is for a.e.~$x$ a good weight for the pointwise ergodic theorem, i.e., whether for every other system $(Y,\nu,S)$ and $g\in L^\infty(Y,\nu)$, the averages 
\[
\aveN f(T^{p(n)}x) g(S^ny)
\]
converge for a.e. $y$, see Assani, Presser \cite[Question 7.1]{AP}.  

A random ergodic theorem with Hardy fields  weights is presented in the recent preprint by Krause, Zorin-Kranich \cite{KZ}.

\smallskip

The paper is organized as follows:\\
In $\S 2$ we recall relevant definitions of Hardy field functions and develop the machinery to prove Theorem \ref{hardy}, as well as provide a uniform estimate for Hardy field weights using the Gowers-Host-Kra uniformity norms of the function $f$, see Theorem \ref{thm:Hardy-seminorms};\\
In $\S 3$ we prove a quantitative estimate on Weyl sums, which we then combine with a metric-entropy argument to prove Theorem \ref{badapp}; and\\
In $\S 4$ we prove Theorem \ref{conv}.

\subsection{Acknowledgements}
This project began during the first author's research visit to the University of California, Los Angeles. Both authors are deeply grateful to Terence Tao for his great encouragement, input and support. 
They also thank Michael Boshernitzan, Kevin Hughes, and Pavel Zorin-Kranich for helpful conversations and comments as well as the referee for improvement suggestions. The support of the Hausdorff Research Institute and the Max Planck Institute for Mathematics in Bonn is also gratefully acknowledged.

\subsection{Notation}
We let $e(t):= e^{2\pi i t}$ denote the exponential.

We will make use of the modified Vinogradov notation. We use $X \lesssim Y$, or $Y \gtrsim X$ to denote the estimate $X \leq CY$ for an absolute constant $C$. If we need $C$ to depend on a parameter,
we shall indicate this by subscripts, thus for instance $X \lesssim_p Y$ denotes
the estimate $X \leq C_p Y$ for some $C_p$ depending on $p$. We use $X \approx Y$ as
shorthand for $Y \lesssim X \lesssim Y $.

We also make use of big-O notation: we let $O(Y)$ denote a quantity that is $\lesssim Y$, and similarly $O_p(Y)$ a quantity that is $\lesssim_p Y$. Finally, we let $\OO(Y)$ denote a quantity which is $\leq Y$.


\section{Wiener-Winter Convergence and Hardy Fields}
In this section we prove our (uniform) Wiener-Wintner type Theorem \ref{hardy} for weights which come from a Hardy field sequence as well as a uniform estimate of the averages (\ref{eq:hardy}) given in Theorem \ref{thm:Hardy-seminorms} below. 
If not explicitly stated otherwise, we will not need the assumption of invertibility. 
We begin with a few preliminary definitions and tools.

\begin{definition}
We call a set $A$ of uniformly bounded complex sequences a \emph{set of (uniform) Wiener-Wintner weights} if for every ergodic measure-preserving system $(X,\mu,T)$ and every $f\in L^1(X,\mu)$ there is $X'\subset X$ with $\mu(X')=1$ such that the averages
\begin{equation}\label{eq:weighted-ave}
\aveN a_n f(T^nx)
\end{equation}
converge (uniformly) for every $x\in X'$ and every
$\{a_n \} \in A$.
\end{definition}
Recall further that for an ergodic measure-preserving system $(X,\mu,T)$, a point $x\in X$ is called \emph{generic} for a function $f\in L^1(X,\mu)$ if it satisfies the assertion of  Birkhoff's ergodic theorem, i.e.,
\[
\lim_{N\to\infty} \aveN T^n f(x) \to \int_X f\, d\mu.
\]

In order to be able to restrict ourselves to a dense subclass of functions we will need the following approximation lemma.

\begin{lemma}[A Wiener-Wintner type Banach Principle]\label{lemma:banach-ww}
Let $A$ be a set of uniformly bounded sequences $\{ a_n \} \subset \C$. Then for every ergodic measure-preserving system $(X,\mu,T)$, the set of functions $f\in L^1(X,\mu)$ so that for almost every $x\in X$, the averages (\ref{eq:weighted-ave}) converge
for every $ \{ a_n \}  \in A$ is closed in $L^1$-norm. The same holds for uniform convergence in $\{ a_n \} $. 
 Moreover, for $f=\lim_{j\to\infty} f_j$ in $L^1(X,\mu)$, the limits of the averages (\ref{eq:weighted-ave}) for $f_j$ converge to the limit of the averages (\ref{eq:weighted-ave}) for $f$.
\end{lemma}
\begin{proof}
Denote
\[ C:=\sup\{\|\{ a_n \} \|_\infty, \  \{ a_n \}  \in A\}\]
and take $(X,\mu,T)$ an ergodic measure-preserving system and $f\in L^1(X,\mu)$.
Observe first that
\[
\left|\aveN a_n T^n f\right| \leq C  \aveN T^n|f|
\]
holds for every $f\in L^1(x,\mu)$.
Assume now that there is a sequence of functions
\[ \{ f_j \}\subset L^1(X,\mu)\]
 with
 \[ \lim_{j\to\infty}\|f_j-f\|_{L^1(X)}=0\]
  so that for every $f_j$ there is a set $X_j'\subset X$ with full measure  such that for every $x\in X_j'$, the averages
\[
\aveN a_nT^n f_j
\]
converge for every $x\in X_j'$ (uniformly) for every $\{ a_n \} \in A$. Define
\[ X' := \bigcap_j X_j' \cap \{ x : x \text{ is generic for each } |f - f_j| \}; \]
we then have $\mu(X')=1$. Now, for every $x\in X'$ and $j\in\N$ we use the triangle inequality and the genericity assumption to majorize
\begin{eqnarray*}
&\limsup_{N,M\to\infty}
\left\{ \sup_{\{a_n \} \in A}\left| \aveN a_n T^nf(x) - \aveM a_n T^nf(x) \right| \right\} \\
&\qquad \leq
\limsup_{N,M\to\infty} \left\{ \sup_{ \{a_n \}\in A} \left| \aveN a_nT^nf_j(x) -\aveM a_nT^nf_j(x)  \right| \right\} \\
&\qquad \qquad  + 2C \cdot \limsup_{N \to \infty} \aveN T^n|f-f_j| (x) \\
& \qquad \qquad \qquad = 2C \|f-f_j\|_{L^1(X)}.
\end{eqnarray*}
Letting $j\to\infty$ finishes the argument. The assertion about the limit follows analogously, see e.g.~\cite[Lemma 21.7]{EFHN}.
\end{proof}

\begin{remark}\label{remark:banach-ww-uniform}
By an inspection of the above argument and the fact that Birkhoff's averages converge uniformly in $x$ for uniquely ergodic systems and continuous functions, one has the following variation of Lemma \ref{lemma:banach-ww}. Let $A$ be a set of uniformly bouded sequences $\{a_n\}\subset \C$ and let $(X,\mu,T)$ be uniquely ergodic. Then the set of all continuous functions for which the averages (\ref{eq:weighted-ave}) converge uniformly in $x\in X$
and $\{a_n \}\in A$ is closed in the $L^1$-norm.
\end{remark}

We will also need the following classical inequality, see e.g.~Montgomery \cite{montgomery}. 
\begin{lemma}[Van der Corput's inequality]\label{lemma:vdc}
Let $N\in \N$ and $u_1,\ldots,u_N\subset \C$ be with $|u_n|\leq 1$ for every $n=1,\ldots,N$. Then for every $H\in\{1,\ldots,N\}$ the following inequality holds.
$$
\left|\aveN u_n\right|^2\leq \frac{2(N+H)}{N^2(H+1)}\sum_{h=1}^H \left(1-\frac{h}{H+1}\right)\left|\sum_{n=1}^{N-h}u_{n+h}\ol{u_n}\right| + \frac{N+H}{N(H+1)}.
$$
\end{lemma}

\smallskip\smallskip


We now introduce Hardy fields and some of their properties. We refer the reader to Boshernitzan \cite{bosh94}, Boshernitzan, Wierdl \cite{bosh/wierdl}, Boshernitzan, Kolesnik, Quas, Wierdl \cite{bosh/kolesnik/quas/wierdl}, Frantzikinakis, Wierdl \cite{frantz/wierdl} and Frantzikinakis \cite{frantz09,frantz10} for further discussion of Hardy field functions and their applications to ergodic theory.

We call two real valued functions of one real variable that are continuous for large values of $s \in \RR$ \emph{equivalent} if they coincide for large $s\in\RR$. Here and later, we say that a property holds for large $s$ (or eventually) if it holds for every $s$ in an interval of the form $[s_0,\infty)$. The equivalence classes under this relation are called \emph{germs}. The set of all germs we denote by $B$ which is a ring.

\begin{definition}
A \emph{Hardy field} is a subfield of $B$ which is closed under differentiation. A Hardy field is called \emph{maximal} if it is maximal among Hardy fields with respect to inclusion of sets. The union of all Hardy fields is denoted by $\U$.
\end{definition}
One can show that every maximal Hardy field contains the class $\mathcal{L}$ of logarithmico-exponential functions of Hardy, i.e., the class of functions which can be obtained by finitely many combinations of real constants, the variable $s$, $\log$, $\exp$, summation and multiplication. Thus, for example, it contains functions of the form $s^\alpha$, $\alpha\in\RR$.

Another property of Hardy fields is that each Hardy field is totally ordered with respect to the order $<_{\infty}$ defined by
$$
f  <_{\infty} g \quad \Longleftrightarrow \quad f(s)< g(s) \quad \text{for all large } s.
$$
Since the class $\mathcal{L}$ belongs to every maximal Hardy field, we conclude that every element of $\U$ is comparable to every logarithmico-exponential  function. In particular, we can define the \emph{type} of a function $p\in U$ to be
$$
t(p):=\inf\{\alpha\in \RR:\, |p(s)|< s^\alpha \text{ for large }s\}.
$$
We say that $p$ is \emph{subpolynomial} if $t(p)<+\infty$, i.e., if $|p|$ is dominated by some polynomial. In particular, for eventually positive subpolynomial $p$ with finite type there is $\alpha\in \RR$ such that for every $\varepsilon$ there is an $s_0$ so that \[ s^{\alpha-\veps}<p(s)<s^{\alpha+\veps}\]
holds for every $s>s_0$. Note that considering eventually positive $p$ is not  a restriction since every nonzero $p\in\U$ is either eventually positive or eventually negative.

We now consider subpolynomial elements of $\U$ with positive non-integer type, such as for example $p(s)=5s^\pi+s\log s.$ More precisely, we introduce following classes.

\begin{definition}
For $\delta\in(0,1/2)$, $M\geq 1$ and $m\in\N_0$ denote by $\M_{\delta, M, m}$ the set of all $p\in \U$ so that there exist $\alpha \in [\delta, 1-\delta]$, $k\leq m$ and $\veps<\min\{(\alpha-\delta)/3, 1-\alpha-\delta\}$ with
\begin{equation}\label{M_delta,M,m}
\frac{1}{M}s^{k+\alpha-\veps-j}\leq p^{(j)}(s)\leq M s^{k+\alpha+\veps-j} \text{ for all }s\geq 1 \text{ and } j=0,\ldots,k+1.
\end{equation}
\end{definition}

\begin{remark}\label{rem:class-M}
By \cite[Lemma 4.2]{bosh/kolesnik/quas/wierdl}, if $a(s),b(s)\in \U$ are non-polynomial with
\[ \lim_{x\to\infty}a(s)/b(s)=0,\]
then $a'(s)$, and $b'(s)$ are non-polynomial and $\lim_{s \to \infty} \frac{a'(s)}{b'(s)} = 0$ too.
Thus, by repeating the argument, every subpolynomial $p\in\U$ of positive non-integer type belongs to some class $\M_{\delta, M, m}$ after a possible left translation.
\end{remark}

We begin our study of the classes $\M_{\delta, M, m}$ in the case where $m=0$;
this special case will anchor the inductive proof of our Theorem \ref{hardy}.
\begin{lemma}\label{lemma:numbers-unif-zero-hardy}
Let $\delta\in(0,1/2)$ and $M\geq 1$.
Then
\[
\lim_{N\to\infty} \sup_{p\in \M_{\delta, M, 0}}\left| \aveN e(p(n))\right|=0.
\]
\end{lemma}
\begin{proof}
Analogously to Kuipers, Niederreiter \cite[Example 2.4]{kuipers/niederreiter} we use Euler's summation formula
\begin{equation}\label{eq:Euler}
\sum_{n=1}^N F(n) = \int_1^N F(t)\, dt + \frac{F(1)+F(N)}{2} + \int_1^N \left(\{t\}-\frac{1}{2}\right) F'(t) \, dt
\end{equation}
for the function $F(t):= e(p(t))$ for $p\in  \M_{\delta, M, 0}$.
We have
\[
\frac{1}{N}\int_1^N |F'(t)|\, dt  \leq  \frac{2\pi M }{N}\int_1^N t^{\delta+\veps -1}\, dt
\]
converging to $0$ uniformly in $p$. We also clearly have
\[
\left| \frac{F(1)+F(N)}{2N} \right|\leq \frac{1}{N}.
\]
To estimate the first summand on the right hand side of (\ref{eq:Euler}) observe
\[
\left(\frac{1}{p'(t)}e(p(t))\right)'=2\pi i  e(p(t)) - \frac{p''(t)}{(p'(t))^2} e(p(t)).
\]
This implies
\begin{eqnarray*}
\left|\frac{1}{N}\int_1^N F(t)\, dt\right| &\leq&
\left(\frac{1}{p'(N)} +\frac{1}{p'(1)}\right) \frac{1}{2\pi N} + \frac{1}{2\pi N} \int_1^N \left|\frac{p''(t)}{(p'(t))^2}\right| \, dt =:I+II.
\end{eqnarray*}
We have
$$
\frac{1}{Np'(N)}\leq \frac{M}{N^{\alpha-\veps}}<\frac{M}{N^{\delta}}
$$
and therefore $I$ converges to $0$ uniformly in $p$. Moreover,
$$
II\leq \frac{M^3}{2\pi N} \int_1^N \frac{t^{\alpha+\veps-2}}{t^{2\alpha-2\veps-2}} \, dt
  \leq \frac{M^3}{2\pi N} \int_1^N t^{-\delta}\, dt.
$$
The assertion follows.
\end{proof}

\begin{remark}\label{remark:bosh}
As pointed out to us by Michael Boshernitzan and follows from \cite{bosh94},
for $p\in \U$ with subpolynomial growth, there are three different types of behavior of $\aveN e(p(n))$.
\begin{itemize}
\item[1)] Assume that there is a rational polynomial $q$ such that $p-q$ is bounded (or, equivalently, has finite limit $c$) at $\infty$. Then one has
$$
\limaveN e(p(n)) = e(c)\limaveN e(q(n))
$$
which exists but does not necessarily equal $0$. For example, for $p$ with negative type the above limit is $1$.
\item[2)] Assume that there is a rational polynomial $q$ so that $(p-q)(s)/\log s$ is bounded (or, equivalently, has finite limit) at $\infty$, but $p-q$ is unbounded. In this case the limit does not exist, see the proofs of \cite[Theorem 1.3]{bosh94} and \cite[Theorem I.2.6]{kuipers/niederreiter} based on the Hardy-Littlewood Tauberian theorem. Hence $(e(p(n)))$ is not a good weight even for the mean ergodic theorem. 
\item[3)] For all other subpolynomial $p\in\U$, the sequence $(p(n))$ is uniformly distributed modulo $1$, see \cite[Theorem 1.3]{bosh94}, and in particular one has 
$$\limaveN e(p(n))=0.$$ 
This is for example the case for $p\in\U$ with finite non-integer type.
\end{itemize}
\end{remark}

We now allow functions to have any positive non-integer type (i.e.\ we allow any $m \geq 0$), and prove our Theorem \ref{hardy}, restated below:

\begin{theorem}\label{Hardythm}
Let $(X,\mu,T)$ be an ergodic \mps, $f\in L^1(X,\mu)$, $\delta\in (0,1/2)$, $m\in \N_0$ and $M\geq 1$. Then there is a subset $X'\subset X$ with $\mu(X')=1$ such that the averages
\begin{equation}\label{eq:ave-alpha}
\sup_{p\in \M_{\delta, M, m}} \left| \aveN e(p(n)) T^nf(x) \right|
\end{equation}
converge to $0$  for every $x\in X'$. Moreover, if $(X,\mu,T)$ is uniquely ergodic and $f\in C(X)$, then one has
\[
\lim_{N\to\infty} \sup_{p\in \M_{\delta, M, m}} \left\| \aveN e(p(n)) T^nf \right\|_\infty =0.
\]
\end{theorem}
In particular, the weights $(e(p(n)))$ are uniform Wiener-Wintner weights for $p\in \M_{\delta, M, m}$ and, if the system is invertible, Wiener-Wintner weights for subpolynomial $p\in \U$ with positive non-integer type. (For the last assertion recall that each such eventually positive $p$ belongs to one of the classes $\M_{\delta, M, m}$ after a possible left translation.)

\begin{proof}
We will argue by induction on $m$ and first discuss the case $m=0$. Take $p\in \M_{\delta,M,0}$, i.e., assume that (\ref{M_delta,M,m}) holds for some
\[ \alpha\in[\delta,1-\delta], \ \veps<\min\{(\alpha-\delta)/3, 1-\alpha-\delta\} \]
and $k=0$.
Let $(X,\mu,T)$ be an ergodic measure-preserving system. Recall the von Neumann decomposition
\begin{equation}\label{eq:von-Neumann-decomp}
L^1(X,\mu)=\C\cdot 1 \oplus \overline{\{f-Tf,\ f\in L^\infty(X,\mu)\}}^{\|\cdot\|_1}.
\end{equation}
For constant $f$, the averages (\ref{eq:ave-alpha}) converge uniformly to $0$ by Lemma \ref{lemma:numbers-unif-zero-hardy}. By (\ref{eq:von-Neumann-decomp}) and Lemma \ref{lemma:banach-ww}, it remains to show the assertion for functions of the form $f-Tf$ for $f\in L^\infty(X,\mu)$.


Let $f\in L^\infty(X,\mu)$ with $\|f\|_\infty\leq 1$  and observe by the telescopic sum argument
\begin{eqnarray*}
\aveN e(p(n)) (T^nf-T^{n+1}f)(x)&=&
\frac{e(p(1)) Tf(x)-e(p(N))T^{N+1}f (x)}{N} \\
&+& \frac{1}{N} \sum_{n=1}^{N-1} (e(p(n+1))-e(p(n)))T^nf (x).
\end{eqnarray*}
Take $x$ with $|T^nf(x)|\leq 1$ for every $n\in\N$. Then in the above, the first term
is bounded by $2/N$
and the second by
\begin{equation}\label{eq:difference}
 \frac{1}{N} \sum_{n=1}^{N-1} |e(p(n+1))-e(p(n))|.
\end{equation}
By the mean value theorem and the assumption on $p$ one has
\[
|e(p(n+1))-e(p(n))|\leq 2\pi \sup_{s\in[n,n+1]} |p'(s)| \leq 2\pi M n^{\alpha+\veps-1}
\]
and therefore (\ref{eq:difference}) is bounded by $2\pi M$ times
\[
 \frac{1}{N} \sum_{n=1}^{N-1}  n^{\alpha+\veps-1} \leq  \frac{1}{N} \sum_{n=1}^{N-1} n^{-\delta}.
\]
The uniform convergence to $0$ follows.

The last assertion of the theorem follows analogously again using the fact that for a uniquely ergodic system $(X,\mu,T)$ and $f\in C(X)$, Birkhoff's ergodic averages converge uniformly in $x$ and Remark \ref{remark:banach-ww-uniform}.

After having established the case $m=0$, assume that the theorem holds for $m\in \{1,\ldots, k-1\}$ and we will show the assertion for $m=k$. Assume that $p$ satisfies (\ref{M_delta,M,m}) for $m=k$.
By Lemma \ref{lemma:banach-ww} we can  assume without loss of generality that $\|f\|_\infty\leq 1$ and take $x$ with $|T^nf(x)|\leq 1$ for every $n\in \N$.
 We are going to use the van der Corput trick from Lemma \ref{lemma:vdc} for $u_n:=e(p(n))f(T^nx)$. Observe first that for $h\in\N$
$$
u_{n+h}\ol{u_n}=e(p(n+h)-p(n))T^n(T^hf\cdot \ol{f})(x).
$$
By Taylor's formula and the assumption on $p$ we have
\[
\aligned
&\left|p(n+h)-p(n)-p'(n)h -\ldots- \frac{p^{(k)}(n) h^{k}}{k!}\right|\\
& \qquad \leq \frac{ h^{k+1}\sup_{s\in [n,n+h]} |p^{(k+1)}(s)|}{(k+1)!}\\
& \qquad \leq \frac{M h^{k+1} n^{\alpha+\veps-1}}{(k+1)!} \\
& \qquad \leq \frac{M h^{k+1} n^{-\delta}}{(k+1)!}.
\endaligned \]

Define $q_h\in U$ by
$$
q_h(s):= p'(s)h +\ldots+ \frac{p^{(k)}(s) h^{k}}{k!}
$$
and observe that $q_h\in M_{\delta,\tilde{M},k-1}$ for some $\tilde{M}$ depending on $\delta, M$ and $h$.
Thus we have
$$
u_{n+h}\ol{u_n}=e(q_{h}(n))T^n(T^hf\cdot \ol{f})(x) + O_{M,k}(h^{k+1} n^{-\delta}).
$$
By Lemma \ref{lemma:vdc}, we thus have, for some constant $C_{M,k}$, depending only on $M$ and $k$, and arbitrary $H, N\in\N$ with $H\leq N$
\begin{eqnarray*}
\left|\aveN u_n\right|^2&\leq& \frac{2(N+H)}{N(H+1)}\sum_{h=1}^H
\left|\frac{1}{N-h}\sum_{n=1}^{N-h}e(q_h(n))T^n(T^hf\cdot \ol{f})(x)\right| \\
&+& C_{M,k} \frac{(N+H)H^k}{N} \aveN n^{-\delta} + \frac{N+H}{N(H+1)}.
\end{eqnarray*}
Now take  $x$ for which additionally the assertion of the theorem is satisfied for functions $T^nf\cdot \ol{f}$ for every $n\in\N$ as well as $\delta$, $m=k-1$ and $\tilde{M}$. Such $x$ form a full measure set by the  induction hypothesis and we conclude that for every $H\in\N$
$$
\limsup_{N\to\infty}\sup_{p\in \M_{\delta, M, k}} \left|\aveN u_n\right|^2
\leq \frac{1}{(H+1)}.
$$
Letting $H\to\infty$ finishes the argument.
\end{proof}

It is natural to ask how restrictive  the class of Hardy weights of positive non-integer type, or the classes $\M_{\delta, M, k}$, are for the (uniform) Wiener-Wintner convergence of the averages (\ref{eq:hardy}). 
For $p\in\U$ with transpolynomial growth, the behavior of $e(p(n))$ can be arbitrarily bad, see Boshernitzan \cite[Theorem 1.6]{bosh94}, so we restrict ourselves to subpolynomial $p\in\U$. Remark \ref{remark:bosh},2) shows that the family $(e(p(n))_{n\in\N})$, $p\in\U$ subpolynomial, is not a Wiener-Wintner family for the ergodic theorem. So one needs some restrictions on the class $p$ to exclude functions from Remark \ref{remark:bosh},2) and, for convergence to \emph{zero}, also to exclude functions from Remark \ref{remark:bosh},1). We remark that by Frantzikinakis \cite[Section 3]{frantWW}, the class of subpolynomial $p\in\U$ with positive non-integer type for which one has convergence of the averages (\ref{eq:hardy}) to zero cannot be enlarged to include all (even quadratic) irrational polynomials.

We now look at this question from a different perspective. 
As mentioned above, one of the problems for general $p$ is possible divergence of $\aveN e(p(n))$ and hence divergence of weighted ergodic averages (\ref{eq:hardy}) for $f=\mathbf{1}$. If we restrict ourselves to a smaller class of functions $f$ than the whole $L^1(X,\mu)$, 
then the class of $p$ can be enlarged, as the following shows. 

Let us consider the following (large) classes: For $m\in\N_0$, $\delta, M>0$ we denote by $\LL_{\delta,M,m}$ the class  of all $p\in\U$ such that  
there exist $k\leq m$  with
\begin{equation}\label{L_delta,M,m}
|p^{(j)}(s)|\leq M s^{k-\delta-j} \text{ for all }s\geq 1 \text{ and } j=0,\ldots,k.
\end{equation}
By Remark \ref{rem:class-M}, every subpolynomial $p\in\U$ belongs to some $\LL_{\delta,M,m}$. 

We recall that  for $f\in L^\infty(X,\mu)$ the \emph{Gowers-Host-Kra} (or \emph{uniformity}) \emph{seminorms} $\|\cdot\|_{U^{m}}$ are defined inductively as follows:
\begin{eqnarray*}
\|f\|_{U^{1}}&:=&\left|\int_X f\,d\mu\right|,\\
\|f\|_{U^{m}}^{2^{m}}&:=&\limsup_{N\to\infty}\aveN \|T^nf\cdot\ol{f}\|_{U^{m-1}}^{2^{m-1}},\quad m\geq 2.
\end{eqnarray*}
For an equivalent definition and properties of the Gowers-Host-Kra seminorms we refer to Host, Kra \cite{HK05} and Eisner, Tao \cite{ET}.

\begin{theorem}[Uniform estimate of averages (\ref{eq:hardy})]\label{thm:Hardy-seminorms}
Let $(X,\mu,T)$ be as above, let $m\in \N_0$ and $\delta, M>0$. 
Then for every $f\in L^\infty(X,\mu)$ the inequality
\[
\limsup_{N\to\infty}\sup_{p\in \LL_{\delta,M,m}}\left|\aveN e(p(n)) T^nf(x)\right| \leq \|f\|_{U^{m+1}}
\]
holds for a.e.~$x\in X$.  Moreover, if $(X,\mu,T)$ is uniquely ergodic and $f\in C(X)$, then one has
\[
\lim_{N\to\infty} \sup_{p\in \LL_{\delta,M,m}} \left\| \aveN e(p(n)) T^nf \right\|_\infty  \leq \|f\|_{U^{m+1}}.
\]
\end{theorem}
\noindent In particular, for every $f$ orthogonal to the Host-Kra factor $\mathcal{Z}_m$, the averages (\ref{eq:hardy}) converge a.e.~to zero 
in $p\in \U$ with type strictly less than $m$ (uniformly in $p\in \LL_{\delta,M,m}$ for every $\delta$ and $M$), and the convergence is uniform in $x$ whenever $(X,\mu,T)$ is uniquely ergodic and $f\in C(X)$.

\begin{proof} 
We proceed by induction in $m$ and assume without loss of generality that $\|f\|_\infty\leq 1$. 

Assume that $p\in \LL_{\delta,M,0}$, i.e., $|p(s)|\leq Ms^{-\delta}$ for every  $s\geq 1$. 
Then we have for every generic $x$ with $|T^nf(x)|\leq 1$  $\forall n\in\N$
\begin{eqnarray*}
\left|\aveN e(p(n)) T^nf(x)\right| &\leq& \aveN |e(p(n))-1| |T^nf(x)| + \left|\aveN T^nf(x)\right| \\
&\leq& 2\pi M\cdot \aveN  n ^{-\delta}  + \left|\aveN T^nf(x)\right| 
\end{eqnarray*}
which converges to $|\int_X f\,d\mu|=\|f\|_{U^1}$ uniformly in $p\in \LL_{\delta,M,0}$.

Assume now that the assertion holds for $m$ and we show that it holds for $m+1$. Take $p\in \LL_{\delta,M,m+1}$ and denote $u_n:=e(p(n)) T^nf(x)$. Take $x\in X$ such that $|T^nf(x)|\leq 1$ holds for every  $n\in\N$. As in the proof of Theorem \ref{Hardythm} observe
\begin{eqnarray*}
u_{n+h}\ol{u_n}&=&e(p(n+h)-p(n)) T^n(T^hf\cdot \ol{f})(x) \\
&=&e(q_h(n)) T^n(T^hf\cdot \ol{f})(x) + O_{M,m}(h^{m+1}n^{-\delta})
\end{eqnarray*}
for $q_h(s)=hp'(s)+\ldots+\frac{h^m p^{(m)}(s)}{m!}$. The function $q_h\in\U$ satisfies $q_h\in \LL_{\delta,\tilde{M},m}$ for some constant $\tilde{M}$ depending on $M$ and $h$. 

By Lemma \ref{lemma:vdc}, we thus have for some constant $C_{M,m}$, depending only on $M$ and $m$, and arbitrary $H, N\in\N$ with $H\leq N$
\begin{eqnarray*}
\left|\aveN u_n\right|^2&\leq& \frac{2(N+H)}{N^2(H+1)}\sum_{h=1}^H\left(1-\frac{h}{H}\right)
\left|\sum_{n=1}^{N-h}e(q_h(n))T^n(T^hf\cdot \ol{f})(x)\right| \\
&+& C_{M,m} \frac{(N+H)H^k}{N} \aveN n^{-\delta} + \frac{N+H}{N(H+1)}.
\end{eqnarray*}
Now take  $x$ for which in addition the assertion of the theorem is satisfied for functions $T^nf\cdot \ol{f}$ for every $n\in\N$ as well as $\delta$, $m$ and $\tilde{M}$. Such $x$ form a full measure set by the  induction hypothesis and we conclude that for every $H\in\N$ using the Cauchy-Schwarz inequality (or the convexity of the function $s\mapsto s^{2^m}$)
\begin{eqnarray}
\limsup_{N\to\infty}\sup_{p\in \M_{\delta, M, k}} \left|\aveN u_n\right|^2
\leq \frac{2}{(H+1)}\sum_{h=1}^H \left(1-\frac{h}{H+1}\right) \|T^hf\cdot \ol{f}\|_{U^{m}} + \frac{1}{H+1}\nonumber\\
\leq \frac{H+1}{H}\left(\frac{2}{H(H+1)}\sum_{h=1}^H \left(H+1-h\right) \|T^hf\cdot \ol{f}\|_{U^{m}}^{2^m}\right)^{1/2^m} + \frac{1}{H+1}.\label{eq:smooth-ave}
\end{eqnarray}
Since for every bounded sequence $\{a_h\}_{h=1}^\infty\subset \C$ and its partial sums $s_h:=\sum_{j=1}^h a_j$, $h\in\N$, we have
$$
\limsup_{H\to\infty} \frac{2}{H(H+1)} \sum_{h=1}^H (H+1-j)a_h
=
\limsup_{H\to\infty} \frac{2}{H(H+1)} \sum_{h=1}^H s_h 
\leq 
\limsup_{H\to\infty}  \frac{s_H}{H},
$$    
letting $H\to\infty$ in (\ref{eq:smooth-ave}) finishes the proof.
\end{proof}

\begin{remark}
\begin{itemize}
\item[1)]
By Lemma \ref{lemma:banach-ww} and the inequality $\|f\|_{U^m}\leq\|f\|_{L^{p_m}}$ for $p_m:=\frac{2^m}{m+1}$, see Eisner, Tao \cite{ET}, one can extend the assertion of Theorem \ref{thm:Hardy-seminorms} to every $f\in L^{p_{m+1}}(X,\mu)$.
\item[2)] 
Theorem \ref{thm:Hardy-seminorms}, the Host-Kra structure theorem, see \cite{HK05}, and Lemma \ref{lemma:banach-ww} imply that the question of finding the largest (or a maximal) family of Wiener-Wiener weights (with limit zero or not) for general ergodic measure-preserving systems restricts to the question of finding such a family for nilsystems. Analogously, a Hardy field weight is a good weight for the pointwise ergodic theorem if and only if it is a good weight for nilsystems. 
\end{itemize}
\end{remark}


\section{Wiener-Wintner Convergence of Twisted Square Means}
In this section we present a Wiener-Wintner type result for weighted polynomial averages for a subclass of badly approximable numbers $\theta$.

It is well-known that the set of badly approximable numbers (i.e., numbers which are badly approximable for some $c>0$) build a perfect compact subset of $[0,1]$ with zero Lebesgue measure and Hausdorff dimension $1$, see Hutchinson \cite{hutchinson}. Moreover, $\theta$ is badly approximable if and only if its coefficients in the continued fraction expansion are bounded. We recall that the sequence $\{ a_j \} \in \N$ is called the continued fraction expansion of $\theta\in (0,1)$ if
$$
\theta=\frac{1}{a_1+\frac{1}{a_2+\frac{1}{a_3+...}}}
$$
holds. There is the following relation between the constant $c$ and the bound of the continued fraction expansion coefficients $M:=\sup\{a_j,\, j\in\N\}$:
$$
\frac{1}{\inf\{a_j,\, j\in\N\}} \leq c \leq \frac{1}{(M+2)(M+1)^2},
$$
see Khintchine \cite[Proof of Theorem 23]{khintchine}.

The starting point in our analysis of Weyl sums (involving badly approximable $\theta$) is as always the classical estimate of exponential sums due to Weyl, see e.g.~Vaughan \cite[Lemma 2.4]{VA}.
\begin{lemma}[Weyl]\label{weyl}
Let $P$ be a polynomial of degree $d$ with leading coefficient $\alpha$, and let $p,q$ be relatively coprime with $|\alpha-p/q|<1/q^2$. Then for every  $\veps>0$ and $N\in \N$,
\[
\left| \sum_{n=1}^N e(P(n)) \right| \lesssim_\veps N^{1+\veps} \left(\frac{1}{q} + \frac{1}{N} + \frac{q}{N^d} \right)^{1/2^{d-1}}.
\]
\end{lemma}

We have the following quantitative estimation of Weyl's sums for square polynomials with non-leading coefficient being $c$-badly approximable, see \cite[Lemma A.5]{frantz/johnson/lesigne/wierdl} for the case of the golden ratio. The proof is a quantification of the argument in \cite{frantz/johnson/lesigne/wierdl}. 

\begin{prop}\label{prop:weyl-badly-appr}
Let $\theta$ be $c$-badly approximable. Then for every $\veps>0$
$$
\sup_{\alpha\in \mathbb{R}}\left| \aveN e(n\theta+n^2\alpha) \right| \lesssim_\veps \frac{1}{c N^{1/32-\veps}}.
$$
\end{prop}
\begin{proof}
Fix $\alpha\in \mathbb{R}$, $N\in \N$ and $\veps>0$. Since the assertion is effective only for $\veps<1/32$, we can assume that  $\veps<1/32$.
The Dirichlet principle implies the existence of $p,q$ relatively prime with $q<N^{2-1/16}$ such that
$$
\left| \alpha-\frac{p}{q} \right| \leq \frac{1}{qN^{2-1/16}}.
$$
If $N^{1/16}\leq q<N^{2-1/16}$, then Weyl's Lemma \ref{weyl} implies that for every $\veps>0$
\[
\left| \aveN e(n\theta + n^2\alpha) \right| \lesssim_\veps N^{\veps}
\left(\frac{1}{q} + \frac{1}{N} + \frac{q}{N^2} \right)^{1/2}\leq  N^{\veps}
\left(\frac{2}{N^{1/16}} + \frac{1}{N} \right)^{1/2} \lesssim \frac{1}{N^{1/32 - \veps}}.
\]
Since $c<1$, the assertion is proved for such $q$.

Thus we may assume from now on that $q<N^{1/16}$. Again the Dirichlet principle applied to $q\theta$ implies the existence of $t,u$ relatively prime with $u\leq N^{1/2}$ such that
\[
\left|\theta - \frac{t}{uq}\right| \leq \frac{1}{uqN^{1/2}}. 
\]
By the assumption on $\theta$ we have on the other hand
\[
\left|\theta - \frac{t}{uq}\right| \geq \frac{c}{q^2u^2}
\]
and we have $cN^{1/2}\leq uq< u N^{1/16}$ implying
\[
u> cN^{7/16}.
\] 

Take $M\in \N$ with $N^{1-1/16}\leq M\leq N$. We now show that the sums
\[
S(M):=\sum_{n=1}^M e(n\theta + n^2 p/q),
\]
where $\alpha$ is replaced by its rational approximation $p/q$, satisfy
\begin{equation}\label{eq:S(M)}
|S(M)|\lesssim_\veps \frac{N^{1+\veps - 3/8}}{c} 
\end{equation}
independently of $M$.

Observe
\begin{equation}\label{eq:S(M)-2}
|S(M)|= \left|\sum_{j=1}^q \sum_{k\geq 0, kq+j\leq M} e((qk+j)\theta + j^2p/q) \right|
\leq \sum_{j=1}^q \left| \sum_{k\geq 0, kq+j\leq M} e((qk+j)\theta) \right|.
\end{equation}
By Weyl's Lemma \ref{weyl} using $|q\theta - t/u|\leq 1/(N^{1/2}u)\leq 1/u^2$, the sum inside satisfies for every $\veps>0$
\[
\left| \sum_{k\geq 0, k\leq (M-j)/q} e((qk+j)\theta) \right| \leq_\veps \left(\frac{M}{q}\right)^{1+\veps} \left( \frac{1}{u} + \frac{q}{M} + \frac{uq}{M} \right).
\]
Remembering that $M\in [N^{1-1/16}, N]$, $q<N^{1/16}$ and $u\in [cN^{7/16},N^{1/2}]$, this is estimated by above by
\[
 \frac{N^{1+\veps}}{q}\left( \frac{1}{cN^{7/16}} + \frac{1}{N^{7/8}}  + \frac{1}{N^{3/8}} \right)
\lesssim \frac{ N^{1+\veps}}{q}  \frac{1}{cN^{3/8}}.
\]
This together with (\ref{eq:S(M)-2}) proves (\ref{eq:S(M)}).

We finally estimate the desired exponential sums using the ``rational'' sums $S(M)$ and summation by parts. Observe
\[ \aligned
\left|\sum_{n=1}^N e(n\theta + n^2\alpha) \right|
&\leq N^{1-1/16} + \left|\sum_{N^{1-1/16}\leq n\leq N} e(n\theta + n^2\alpha)  \right|\\
&=  N^{1-1/16} + \left|\sum_{N^{1-1/16}\leq n\leq N} e(n^2(\alpha-p/q)) [S(n)-S(n-1)]  \right|
\endaligned \]
Since the discrete derivative of $ e(n^2\gamma)$ satisfies
\[
| e((n+1)^2\gamma) -  e(n^2\gamma)| =|e((2n+1)\gamma) -1|\lesssim n|\gamma|,
\]
summation by parts and (\ref{eq:S(M)}) implies
\[\aligned
\left|\sum_{n=1}^N e(n\theta + n^2\alpha) \right|
& \lesssim N^{1-1/16} + |S(N)| + |S([N^{15/16}]+1)| + \sum_{N^{1-1/16}\leq n\leq N} n |\alpha-p/q| |S(n)|\\
&\leq  N^{1-1/16} + |S(N)| + |S([N^{15/16}]+1)| + \sum_{N^{1-1/16}\leq n\leq N}  |S(n)| N^{1/16-1}\\
& \lesssim_\veps N^{1-1/16} + N  \frac{N^{1+\veps-3/8}}{cN^{1-1/16}} \\
&= N^{1-1/16} + c^{-1} N^{1+\veps-3/16} \\
&\lesssim c^{-1}  N^{1-1/16},
\endaligned\]
since by assumption $\veps \leq 1/32$. We used that
\[ n|\alpha-p/q| \leq \frac{N}{qN^{2-1/16}} < N^{1/16 -1} \]
in passing to the second line.

Thus we have
\[
\sup_{\alpha\in\RR}\left|\aveN e(n\theta + n^2\alpha) \right|  \lesssim_\veps \max\left\{ \frac{1}{N^{1/32-\veps}}, \frac{1}{cN^{1/16}}\right\},
\]
completing the proof.
\end{proof}

Recall that the \emph{upper Minkowski (or box) dimension} of a set $E\subset \RR$ is given by
$$
\overline{\dim}_\text{box}(E):=\limsup_{r\to\infty}\frac{N(\veps)}{\log (1/\veps)},
$$
where $N(\veps)$ is the minimal number of intervals with length $\veps$ needed to cover $E$. The upper Minkowski dimension of a set is always bigger  than or equal to its Hausdorff dimension. It is well-known that one can replace $N(\veps)$ in the above definition by the so-called  \emph{$\veps$-metric entropy number}, i.e., the cardinality of the largest $\veps$-net in $E$, where an  \emph{$\veps$-net} is a set with distances between any two different elements being larger than or equal to $\veps$. See Tao \cite[Section 1.15]{tao-epsilon} for basic properties of the (upper) Minkowski dimension.


We are ready for our uniform Wiener-Wintner type result for subsets of  $c$-badly approximable numbers with small Minkowski dimension. We restate our result below for the reader's convenience:

\begin{theorem}\label{thm:unif-WW-small-dim}
Let $0 < c <1$ and $E = E_c \subset [0,1]$ be a set of $c$-badly approximable numbers with upper Minkowski dimension strictly less than $1/16$.

For every $f \in L^2(X)$ there exists a subset $X_f \subset X$ of full measure so that the averages
\[ \left\{ \frac{1}{N} \sum_{n \leq N} e(n\theta) T^{n^2} f(x) \right\} \]
converge \emph{uniformly} (in $\theta \in E_c$) to zero for all $x \in X_f$, where $X_f$ is independent of $\theta \in E_c$.
\end{theorem}

\begin{proof}
We first prove the result for $f$ being a simple function, i.e., a finite linear combination of characteristic functions of measurable sets. By the boundedness of $f$ and Rosenblatt, Wierdl \cite[Lemma 1.5]{RW}, it is enough to prove that for any lacunary constant $\rho > 1$
\[
 \lim_{N \to \infty, \ N= \lfloor \rho^k \rfloor} \sup_{\theta\in E} \left|\aveN e(n\theta) T^{n^2}f(x)\right|=0
\]
holds for almost every $x\in X$. Note that for a fixed $N$, the function defined by
\[
g_{N}(x):=\sup_{\theta\in E} \left|\aveN e(n\theta) T^{n^2}f(x)\right|
\]
is again a simple function and hence measurable. By the Borel-Cantelli lemma, it is enough to show that
\[
\sum_{N= \lfloor \rho^k \rfloor} \left\| \sup_{\theta\in E} \left|\aveN e(n\theta) T^{n^2}f(x)\right| \right\|_{L^2(X)} < \infty;
\]
this will be accomplished by showing that, provided
\[
\overline{\dim}_{\text{box}}(E) < 1/16,
\]
we have
\begin{equation}\label{eq:norm-of-sup}
\left \| \sup_{\theta\in E} \left|\aveN e(n\theta) T^{n^2}f(x)\right| \right\|_{L^2(X)} = O(N^{-\nu})
\end{equation}
for some $\nu = \nu( \overline{\dim}_{\text{box}}(E) ) > 0$.

Assume without loss of generality that $\|f\|_\infty\leq 1$. Note first that for a fixed $\theta\in E$,  the spectral theorem implies
\[\aligned
\left\| \aveN e(n\theta) T^{n^2}f(x) \right\|_{L^2(X)}^2
&= \int_{[0,1]} \left| \aveN e(n\theta + n^2\alpha)  \right| ^2\, d\mu_f(\alpha)\\
&\leq \sup_{\alpha \in [0,1]} \left|\aveN e(n\theta + n^2 \alpha) \right|^2
\endaligned \]
for the corresponding spectral measure $\mu_f$ on $[0,1]$. Thus by Proposition \ref{prop:weyl-badly-appr} we have
\begin{equation}\label{eq:norm-fixed-theta}
\left\| \aveN e(n\theta) T^{n^2}f(x) \right\|_{L^2(X)}   \lesssim_\veps \frac{1}{c N^{1/32-\veps}} \quad \text{for every } \theta\in E.
\end{equation}

To show (\ref{eq:norm-of-sup}), suppose $\gamma < 1$ is such that
\[
\overline{\dim}_{\text{box}}(E) = \gamma/16 < 1/16,
\]
and choose $\beta > 1$ but so near to it that $\gamma \beta =: \kappa < 1$.
Now, let $\Delta_N$ be a maximal $\frac{1}{10N^{\beta}}$-net in $E$, i.e., a set of maximal cardinality of points in $E$ with distance between any two distinct points being larger than $\frac{1}{10N^{\beta}}$.

Since $\overline{\dim}_{\text{box}}(E) = \gamma/16$,  we remark that for all $N$ sufficiently large
\[ \log |\Delta_N| \leq \overline{\dim}_{\text{box}}(E) \cdot \log (N^{\beta}) =  \log (N^{\gamma\beta/16}),
\]
or
\begin{equation}\label{eq:Delta}
|\Delta_N| \leq  N^{\kappa /16}.
 \end{equation}

Now, take $x\in X$ such that $|T^{n^2}f(x)|\leq \|f\|_\infty\leq  1$ holds for every $n\in \N$; note that the set of such $x$ has full $\mu$-measure. The polynomial $p_N$ defined by
\[
p_N(\theta) := \aveN e(n\theta) T^{n^2}f(x)
\]
satisfies $|p_N(\theta)| \leq 1$ and by inspection (or Bernstein's polynomial inequality \cite{Be})
\[
\|p_N'\|_\infty \leq  2\pi N.
\]

Take now $\theta\in E$ and denote by $\tau$ the nearest element to $\theta$ from $\Delta_N$. By the maximality of $\Delta_N$ we have
\[
|\theta-\tau|\leq \frac{1}{10N^{\beta}}.
\]

By the mean-value theorem we may therefore estimate
\[ \aligned
|p_N(\theta)| &\leq |p_N(\tau)| + |\theta - \tau| 2\pi N \\
&\leq \max_{\theta' \in \Delta_N} |p_N(\theta')| + N^{1-\beta} \\
&\leq \left( \sum_{\theta' \in \Delta_N} |p_N(\theta')|^2 \right)^{1/2}  + N^{1-\beta},
\endaligned\]
replacing the $l^\infty$-norm with the larger $l^2$-norm.

We may consequently estimate using (\ref{eq:norm-fixed-theta}) and (\ref{eq:Delta})
\[ \aligned
\left\| \sup_{\theta \in E} \left| \aveN e(n\theta) T^{n^2}f(x)\right| \right\|_{L^2(X)}
&\leq
\left\| \left( \sum_{\theta' \in \Delta_N} |p_N(\theta')|^2 \right)^{1/2} \right\|_{L^2(X)} + N^{1-\beta} \\
&= \left( \sum_{\theta' \in \Delta_N} \|p_N(\theta')\|_{L^2(X)}^2 \right)^{1/2} + N^{1-\beta} \\
&\lesssim_\veps c^{-1}| \Delta_N |^{1/2} \cdot N^{-1/32+\veps} + N^{1-\beta} \\
&\leq c^{-1}N^{\kappa /32} \cdot N^{-1/32+\veps} + N^{1-\beta} \\
&\leq c^{-1}N^{\frac{\kappa-1}{32} +\veps} + N^{1- \beta}. \\
\endaligned\]
Choosing $\veps<\frac{1-\kappa}{32}$ finishes the proof of (\ref{eq:norm-of-sup}).

This concludes the argument in the case where $f$ is a simple function.


To extend the result to all of $L^2(X)$, we observe that the maximal function satisfies
\[
\sup_{\theta \in E}  \left|\aveN e(n\theta) T^{n^2} f(x)\right| \leq \M_{N,\text{sq}} f := \aveN T^{n^2} |f|(x),
\]
where the maximal function for squares on the right hand side is $L^2$-bounded by the celebrated result of Bourgain \cite{B0}.

Take now $f \in L^2(X)$ is arbitrary, and $g$ a simple function. 
For any $\epsilon > 0$, we have the containment
\[ \aligned
& \left\{ x : \limsup_N \ \sup_{\theta \in E} \left| \aveN e(n\theta) T^{n^2} f(x) \right| > \epsilon \right\} \\
& \qquad \subset
\left\{ x: \limsup_N \ \sup_{\theta \in E} \left| \aveN e(n\theta) T^{n^2} g (x) \right| > \epsilon/2 \right\} \cup
\left\{ x : \limsup_N \M_{N,\text{sq}} (f- g) (x)  > \epsilon/2 \right\} \\
& \qquad \subset \left\{ x : \sup_N \M_{N, \text{sq}} (f-g)(x)  > \epsilon/2 \right\}
\endaligned \]
by the previous considerations for simple functions. 
We thus have by Bourgain's maximal inequality the upper estimate
\[ \aligned
&\mu \left( \left\{ x : \limsup_N \ \sup_{\theta \in E} \left|\aveN e(n\theta) T^{n^2} f(x) \right| > \epsilon \right\} \right) \\
&\qquad \leq \mu \left( \left\{ x : \sup_N \M_{N,\text{sq}} (f-g)(x)  > \epsilon/2 \right\} \right) \\
&\qquad \leq \frac{4}{\epsilon^2} \left\|\sup_N \M_{N,\text{sq}} (f-g) \right\|_{L^2(X)}^2  \\
&\qquad \lesssim \frac{1}{\epsilon^2} \|f-g\|_{L^2(X)}^2. 
\endaligned\]
Since the last quantity can be made as small as we wish (independent of $\epsilon$) the result follows.
\end{proof}

\begin{ex}[Sets of badly approximable numbers with small Minkowski dimension]
Let $A=\{a_1,\ldots,a_k\}\subset \N$ and denote by $E_A$ the set of all $\theta\in [0,1]$ which coefficients in the continued fraction expansion all belong to $A$. By Jenkinson \cite{jenkinson}, the set $\{\dim_H(E_A)\}$ of Hausdorff's dimensions of the sets $E_A$ is dense in $[0,1/2]$ (in fact, it is even dense in [0,1] as showed by  Kesseb\"ohmer, Zhu \cite{texan-conj}). Moreover, if $1\notin A$, then the Hausdorff and the Minkowski dimension of $E_A$ concide, see Falconer \cite[Theorem 23]{falconer}. Therefore, infinitely many sets $E_A$ satisfy the condition in Theorem \ref{thm:unif-WW-small-dim}.
\end{ex}
%


\section{Pointwise Convergence of the Twisted Polynomial Means}

In this section we consider the behavior of the twisted means corresponding to more general polynomial shifts. We also relax our \emph{badly approximable} hypothesis and now allow $\theta$ to be arbitrary. Although establishing a Wiener-Wintner type theorem for arbitrary $\theta$ seems very difficult under these general assumptions, we are able to establish pointwise convergence for such means. Specifically, we prove the following

\begin{theorem}\label{main}
For any measure-preserving system, $(X,\mu,\tau)$, any $\theta \in [0,1]$, and any polynomial $P(n)$ with integer coefficients, the twisted means
\[ M_t^\theta f(x):= \frac{1}{t}\sum_{n \leq N} e(n \theta) \tau^{P(n)} f(x), \]
converge $\mu$-a.e.\ for any $f \in L^p(X), \ p >1$,
\end{theorem}

\subsection{Strategy}
Our strategy is as follows:

As previously remarked, using the majorization
\[ \sup_t | M_t^\theta f| \leq \M_{P} |f| := \sup_t \frac{1}{t}\sum_{n \leq t} \tau^{P(n)} |f|(x),\]
we see that the set of functions in $L^p, \ p>1$ for which $\mu$-a.e. convergence holds is closed, since Bourgain's square polynomial maximal function $\M_{P}$ is bounded on $L^p$ \cite{B1}.

Consequently, in proving convergence we may work exclusively with bounded $f \in L^\infty$.

For such functions, it suffices to prove that for any $\rho > 1$, the means
\[ \{ M_t^\theta f  : t \in \lfloor \rho^{\N} \rfloor \} =: \{ M_t^\theta f : t \in I_\rho \} \]
converge $\mu$-a.e.; we establish this result through a (long) \emph{variational estimate} on the $\{M_t^\theta f\}$:

\begin{definition}
For $0 < r <\infty$, we define the $r$-variation of the means $\{M_t^\theta f\}$
\[ \V^r( M_t^\theta f )(x):= \sup_{(t_k) \text{ increasing}} \left( \sum_k |M_{t_k}^\theta f - M_{t_{k+1}}^\theta f |^r \right)^{1/r}(x). \]
(The endpoint $\V^\infty (M_{t}^\theta f)(x) := \sup_{t,s} |M_t^\theta f - M_s^\theta f|(x)$ is comparable to the maximal function $\sup_t |M_t^\theta f|(x)$, and so is typically not introduced.)
\end{definition}

By the nesting of little $l^p$ spaces, we see that the variation operators grow more sensitive to oscillation as $r$ decreases. (We shall restrict our attention to the range $2<r<\infty$.)
This sensitivity is reflected in the fact that
although having bounded $r$-variation, $r<\infty$, is enough to imply pointwise convergence,
there are collections of functions which converge, but which have unbounded $r$ variation for \emph{any} $r< \infty$. (e.g.\ $\{(-1)^i\frac{1}{\log i+1} \}$)

To prove that the $\V^r(M_t^\theta f) < \infty$ converges almost everywhere, we will show that they are \emph{bounded} operators on $L^2(X)$. In particular, we will prove the following

\begin{proposition}
There exists an absolute $C_{r,\rho,\theta,P}$ so that
\[ \| \V^r (M_t^\theta f : t \in I_\rho) \|_{L^2(X)} \leq C_{r,\rho,\theta,P} \|f \|_{L^2(X)}\]
provided $r> 2$.
\end{proposition}

By the transference principle of Calder\'{o}n \cite{C}, this result will follow from the analogous one on the integer lattice:

\begin{proposition}\label{key}
Suppose $r> 2$. Then there exists an absolute constant $C_{r,\rho,\theta,P}$ so that for any $f \in l^2(\Z)$,
\[ \| \V^r(K_N^\theta*f)\|_{l^2(\Z)} \leq C_{r,\rho,\theta,P} \|f\|_{l^2(\Z)},\]
where
\[ K_N^\theta*f(x) := \frac{1}{N} \sum_{n \leq N} e(-n\theta)f(x - P(n))\]
is the discrete convolution operator.
\end{proposition}
\begin{remark}
For notational ease, we have chosen to work with the above convolution kernels; by replacing $\theta$ and $P$ with $-\theta$, $-P$, we prove an analogous result for
\[ \bar{K}_N^\theta*f(x) := \frac{1}{N} \sum_{n \leq N} e(n\theta)f(x + P(n)) \]
which may be transferred appropriately.
\end{remark}


It is this proposition to which we now turn.

\subsection{Preliminaries}
Fix throughout $0 < \delta \ll 1$.

Let $\rho > 1$ be a temporarily fixed lacunarity constant, and set
\[ I = I_\rho:= \{ \lfloor \rho^k \rfloor : k \in \N \};\]
we will enumerate such elements using capital $N, M, K$, etc.

This is an $L^2$ problem, so we will use the Fourier transform. Our strategy will be to replace the twisted multipliers
\[ \aligned
\widehat{K_N}(\alpha) &:= \frac{1}{N} \sum_{n \leq N} e(P(n)\alpha - n \theta) \\
&= \frac{1}{N} \sum_{n \leq N} e( m_d \alpha \cdot n^d + \dots + m_d \alpha \cdot n^2 + (m_1 \alpha - \theta) \cdot n ) \endaligned \]
with increasingly tractable families of multipliers which remain $L^2$-close. \footnote{Our projection-based approach is based off of the excellent exposition of \cite{RW}.}

The arithmetic properties of $\theta$ figure centrally in this approximation, so to organize our approach we introduce \emph{$N$-$\theta$ rational approximates};
their significance is that the so-called $N$-major boxes for $\widehat{K_N}(\alpha)$ (introduced below) only appear ``arithmetically near'' such approximates.

\subsection{$N$-$\theta$ Rational Approximates}

Regarding $\theta \in [0,1]$ as given, for $N \in I_{\rho} = I$ we define the $N$-$\theta$ rational approximates $\{ \frac{x_N}{y_N} \} \subset [0,1]$ to be rational numbers (in reduced form) with
\begin{itemize}
\item $y_N \leq m_d N^\delta$;
\item $|\gamma_N(\theta)| = |\gamma_N| := |\frac{x_N}{y_N} - \theta| \leq 2 N^{\delta - 1}$.
\end{itemize}

Some remarks are in order:

For many $N$ there need not exist $N$-$\theta$ rational approximates, but -- for sufficiently large $N$ -- such $N$-$\theta$ approximates are unique if they exist:
if $x/y, \ p/q$ were distinct $N$-$\theta$ approximates, we would have
\[ \left( \frac{1}{m_dN^\delta} \right)^{2} \leq \frac{1}{yq} \leq \left| x/y - p/q \right| \leq
\left| x/y - \theta \right| + \left| \theta - p/q \right| \leq 4 N^{\delta - 1},\]
for the desired contradiction.

Many $N$ may share the same $N$-$\theta$ rational approximate; for example, if $\theta$ is very close to $0$, then $0/1$ will serve as an $N$-$\theta$ approximate many times over.

To make matters clearer, we enumerate the \emph{distinct} $N$-$\theta$ rational approximates $\{ \frac{x_{N_j}}{y_{N_j}} : j\}$ according to the size of pertaining $N_j$.
Note that if $\theta$ is rational, or more generally \emph{badly approximable}, there are only finitely many distinct $N$-$\theta$ rational approximates.

Moreover, distinct $N$-$\theta$ approximates are necessarily sparsely spaced.

Suppose $N< N_j$; then arguing as above we see
\[ \frac{1}{m_d N^\delta y_{N_j}} < |x_N/y_N - x_{N_j}/y_{N_j}| < 4N^{\delta-1},\]
and thus $N^{1-2\delta} \lesssim_P y_{N_j}$.

\subsection{Major Boxes}
For $a/b, p_{d-1}/q_{d-1}, \dots, p_2/q_2, x/y \in [0,1]$ with
\[
\lcm(b,q_{d-1},\dots,q_2,y) \leq N^\delta, \]
 we define the $N$-major box
\[ \aligned
&\MM_N(a/b, p_{d-1}/q_{d-1}, \dots, p_2/q_2, x/y) \\
& \qquad := \left\{ (\alpha_d,\alpha_{d-1},\dots,\alpha_2,\alpha_1): |\alpha_d - a/b| \leq N^{\delta - d}, \
|\alpha_i - p_i/q_i| \leq N^{\delta - i}, \ |\alpha_1 - x/y| \leq N^{\delta - 1} \right\}. \endaligned \]
By Bourgain's \cite[Lemma 5.6]{B1}, we know that if $(\alpha_d,\dots,\alpha_1)$ does not lie in some major box, there exists some (small) $\kappa > 0$ so that
\[ \left| \frac{1}{N} \sum_{n \leq N} e( \alpha_d n^d + \dots + \alpha_1 n) \right| \lesssim N^{-\kappa}.\]

In our present context, this means that
\[ \widehat{K_N}(\alpha) := \frac{1}{N} \sum_{n \leq N} e( m_d \alpha \cdot n^d + \dots + m_2 \alpha \cdot n^2 + (m_1 \alpha - \theta) \cdot n ) \]
is $O(N^{-\kappa})$ unless there exist some tuple
\[ (a/b, p_{d-1}/q_{d-1}, \dots, p_2/q_2, x/y ), \ \lcm(b, q_{d-1},\dots,q_2,y) \leq N^{\delta},\]
so
\[ \aligned
m_d \alpha &\equiv a/b + \OO(N^{\delta - d}) \mod 1, \\
m_i \alpha &\equiv p_i/q_i + \OO(N^{\delta - i}) \mod 1, \ 2 \leq i \leq d-1, \ \text{ and} \\
m_1 \alpha - \theta &\equiv x/y + \OO(N^{\delta - 1}) \mod 1. \endaligned \]
The first point means that there exists some $0\leq j \leq m_d -1$
\[ \alpha \equiv \frac{1}{m_d} (j + a/b) + \OO\left(\frac{1}{m_d}N^{\delta - d} \right) \mod 1\]
which in turn forces
\[ m_i \alpha \equiv \frac{m_i}{m_d} ( j + a/b) + \OO\left(\frac{m_i}{m_d} N^{\delta - d} \right) \mod 1\]
and thus, for $2 \leq i \leq d-1$,
\[ \frac{m_i}{m_d}(j+a/b) \equiv p_i/q_i + \OO(2 N^{\delta -1}) \mod 1\]
for $N$ sufficiently large.
Since the left side of the foregoing has a priori denominator $\leq m_d N^{\delta}$, while the right side has denominator $\leq N^{\delta}$, we in fact must have
\[\frac{m_i}{m_d}(j+a/b) \equiv p_i/q_i \mod 1. \]

Turning to the $i=1$ term: we have on the one hand
\[ m_1 \alpha \equiv \frac{m_1}{m_d}(j+a/b) + \OO\left(\frac{m_1}{m_d} N^{\delta - d} \right) \mod 1 \]
while on the other hand
\[ m_1 \alpha - \theta  \equiv x/y + \OO(N^{\delta - 1}) \mod 1.\]
This forces (for $N$ sufficiently large)
\[ \frac{m_1}{m_d}(j+a/b) \equiv x/y + \theta + \OO( 2N^{\delta -1 } ) \mod 1 \]
and thus
\[ \frac{m_1}{m_d}(j+a/b) - x/y \equiv \theta + \OO( 2N^{\delta -1 } ) \mod 1.\]
Since the left hand side of the above expression has denominator of size $\leq m_d N^{\delta}$, it must be $\equiv x_N/y_N \mod 1$, i.e.\ we must have
\[ \frac{m_1}{m_d}(j+a/b) - x/y \equiv x_N/y_N \mod 1. \]
This forces
\[ \frac{m_1}{m_d}(j+a/b) - x_N/y_N \equiv x/y \mod 1.\]
The upshot is that, beginning with the assumption that
\[ (m_d \alpha, \dots, m_2 \alpha, m_1 \alpha - \theta) \in \MM_N(a/b,p_{d-1}/q_{d-1},\dots,p_2/q_2,x/y) \]
for \emph{some} major box, with
\[ \alpha \equiv \frac{1}{m_d} (j + a/b) + \OO\left(\frac{1}{m_d}N^{\delta - d} \right) \mod 1,\]
we have deduced the relationship
\[ \aligned
 \frac{p_i}{q_i} &\equiv \frac{m_i}{m_d}(j+a/b) \mod 1 \\
\frac{x}{y} &\equiv \frac{m_1}{m_d}(j+a/b) - x_N/y_N \mod 1. \endaligned\]

\begin{remark}
If $\theta$ is badly approximable, so there are only finitely many $N$-major boxes, then for all $N$ large we see that the twisted polynomial means are $O(N^{-\kappa})$.
By arguing as in the previous section, uniform Wiener-Wintner theorems can be proven for the twisted polynomial means
for badly approximable $\theta$ which live in sets of upper Minkowski dimension $< 2\kappa$. We will not focus on attaining the best possible numerical constant $\kappa$ in our present context.
\end{remark}

In fact, a little more is true: for $\alpha \in \MM_N(a/b,p_{d-1}/q_{d-1},\dots,p_2/q_2,x/y)$ we knew a priori that for $2 \leq i \leq d$
\[ m_i \alpha \equiv p_i/q_i + \OO(N^{\delta - i}) \mod 1;\]
using the above determined structure, i.e.\ $\alpha \equiv \frac{1}{m_d}(j+a/b) + \OO\left(\frac{1}{m_d}N^{\delta - d}\right)$,
we in fact find that
\[ m_i \alpha \equiv \frac{m_i}{m_d}(j+ a/b) + \OO\left(\frac{m_i}{m_d} N^{\delta - d}\right) \mod 1.\]

We're now ready for a working characterization of our major boxes.

For $a/b \in \QQ \cap [0,1]$ in reduced form, define by equivalence $\mod 1$
\[ \aligned
\frac{ \widehat{a_i^j}}{ \widehat{b_i^j} } &:\equiv \frac{m_i}{m_d}(j+a/b) \ \text{ for } 2 \leq i \leq d-1 \\
\frac{ \widehat{a_{1,N}^j}}{ \widehat{b_{1,N}^j} } &: \equiv \frac{m_1}{m_d}(j+a/b) - x_N/y_N, \endaligned \]
and set
\[ b_N^j := \lcm (b, \widehat{b_{d-1}^j}, \dots, \widehat{b_2^j}, \widehat{ b_{1,N}^j } ).\]
Collect all $b_N^j \leq N^\delta$ in
\[ \aligned
A_N^0 &:= \{ a/b \neq 0/1 : b_N^0 \leq N^\delta \} \\
A_N^j &:= \{ a/b : b_N^j \leq N^\delta \}, \ 1 \leq j \leq m_{d-1}, \endaligned \]
(these sets are possibly empty) and set
\[ A_N := \bigcup_{j=0}^{m_d -1} A_N^j.\]

If we define
\[
\MM_N^j(a/b) := \left\{ \alpha : \alpha \equiv \frac{1}{m_d}(j + a/b ) + \OO \left(\frac{1}{m_d} N^{\delta - d} \right) \mod 1 \right\} \]
then we have shown that the union of the major boxes
\[ \bigcup_{ (a/b,p_{d-1}/q_{d-1},\dots,p_2/q_2,x/y) : \lcm (b,q_{d-1}, \dots, q_2,y) \leq N^{\delta} }\MM_N(a/b,p_{d-1}/q_{d-1},\dots,p_2/q_2,x/y) \]
is (simply) contained in
\[ \bigcup_{j=0}^{m_d -1 } \bigcup_{a/b \in A_N^j} \MM_N^j (a/b).\]

Before proceeding, we further decompose our sets $A_N^j, A_N$ according to the size of denominators of its elements. Specifically, for $2^t \leq N^\delta$, we define
\[ A_{N,t}^j := \{ a/b : b_N^j \approx 2^t\},\]
and $A_{N,t} := \bigcup_{j=0}^{m_d -1} A_{N,t}^j$. (Note that $|A_{N,t}| \lesssim_P 4^t$.)

We isolate the following simple subdivision lemma:

\begin{lemma}
For each $t$, there is at most one scale $l(t)$ so that if $A_{N,t}^j \neq \emptyset$ for any $0\leq j \leq m_d -1$, we necessarily have $N_{l(t)} \leq N < N_{l(t)+1}$.
\end{lemma}
\begin{proof}
Seeking a contradiction, suppose there existed some $N_k < N_l$ so that
\[ N_k \leq N < N_{k+1} \leq N_l \leq M < N_{l+1},\]
and $2^t \leq N^\delta < M^\delta$, but
\[ A_N^j, \ A_M^h \neq \emptyset \]
for some $0 \leq j,h \leq m_d -1$.
This means we may find some $a/b \in A_N^j$ so that
\[ \widehat{b_{1,N}^j} \leq b_N^j \lesssim 2^t \leq N^\delta,\]
which forces $y_{N_k} = y_N \lesssim m_d 2^t$.

Similarly, we may find some $p/q \in A_M^h$ so that
\[ \widehat{q_{1,M}^h} \leq q_M^h \lesssim 2^t \leq N^\delta,\]
which forces $y_{N_l} = y_M \lesssim m_d 2^t$, which means we have
\[ y_{N_l} \lesssim m_d 2^t \leq m_d N^\delta.\]
But this contradicts the sparsity condition, $N^{1-2\delta} \lesssim_P y_{N_l}$, and completes the proof.
\end{proof}

\subsection{The Multiplier on Major Arcs}
We now proceed to describe the shape of the twisted multiplier
\[ \widehat{K_N}(\alpha) = \frac{1}{N} \sum_{n \leq N} e( m_d \alpha \cdot n^d + \dots + m_d \alpha \cdot n^2 + (m_1 \alpha - \theta) \cdot n )\]
on $\MM_N^j(a/b)$.

For $a/b \in A_N^j$ define the weight
\[ S_N^j(a/b):= \frac{1}{b_N^j} \sum_{r=1}^{b_N^j} e \left(r^d \frac{a}{b} + r^{d-1} \frac{ \widehat{ a_{d-1}^j} }{ \widehat{ b_{d-1}^j} } + \dots + r
\frac{ \widehat{ a_{1,N}^j} }{ \widehat{ b_{1,N}^j} } \right) ;\]
we remark that
\[ S_N^0(0/1) = \frac{1}{y_N} \sum_{r=1}^{y_N} e( -r x_N/y_N ). \]

We note that for any fixed $0 \ll \nu < 1/d$ we have
\[ |S_N^j(a/b)| \lesssim (b_N^j)^{-\nu} \]
by Hua's \cite[\S 7, Theorem 10.1]{HUA}.

Define further the oscillatory (twisted) ``pseudo-projection''
\[ V_N(\beta) := \int_0^1 e( N^d m_d t^d \beta + Nt\gamma_N) \ dt,\]
and observe that
\[ \left| V_N(\beta) - \int_0^1 e( Nt\gamma_N) \ dt \right| =: \left| V_N(\beta) - \omega_N \right| \lesssim N^d m_d |\beta| \]
by the mean-value theorem, while
\[ |V_N(\beta)| \lesssim \frac{1}{N m_d^{1/d} |\beta|^{1/d}} \]
by van der Corput's estimate on oscillatory integrals \cite[\S 8]{S}.

We have the following approximation lemma:

\begin{lemma}
On $\MM_N^j(a/b)$, if $a/b \in A_N^j$
\[ \widehat{K_N}(\alpha) = S_N^j(a/b) V_N \left( \alpha - \frac{1}{m_d}(j+a/b) \right) + O(N^{2\delta - 1}). \]
If $|\alpha| \leq \frac{1}{m_d} N^{\delta - d}$,
\[ \widehat{K_N}(\alpha) = S_N^0(0/1) V_N( \alpha ) + O(N^{2\delta - 1}).\]
\end{lemma}
\begin{proof}
This follows by arguing as in the proof of \cite[Proposition 4.1]{BK}.
\end{proof}

Let  $\chi:= 1_{[-1,1]}$ denote the indicator function of $[-1,1]$.
If we define via the Fourier transform
\[ \aligned
\widehat{^{0}R_N}(\alpha) &:= S_N^0(0/1) V_N(\alpha) \chi(N^{d-\delta}m_d \alpha) \\
\widehat{R_N}(\alpha) &:= \sum_{j=0}^{m_d-1} \sum_{a/b \in A_N^j} S_N^j(a/b) V_N \left( {\alpha - \frac{1}{m_d}(j+a/b)} \right) \chi\left(N^{d-\delta} m_d \left({\alpha - \frac{1}{m_d}(j+a/b)}\right)\right) \endaligned \]
then in light of the previous lemma we have
\[ \widehat{K_N}(\alpha) = \widehat{^{0}R_N}(\alpha) + \widehat{R_N}(\alpha) + O(N^{2\delta - 1}),\]
and thus
\[ \V^r(K_N*f) \leq \V^r\left( ^{0}R_N*f \right) + \V^r( R_N*f) + \left( \sum_N \left|\left(  ^{0}R_N*f+ R_N*f \right) - K_N*f\right|^2 \right)^{1/2}.\]
By Parseval's inequality, the square sum has $l^2$ norm $\lesssim_\rho 1$, so we need only $l^2$-bound
\[  \V^r\left( ^{0}R_N*f \right)  \]
and
\[\V^r( R_N*f) .\]

The following transference lemma, essentially due to Bourgain \cite[Lemma 4.4]{B1} will allow us to view our operators as acting on the real line $\RR$:

\begin{lemma}
Suppose $r >2$, and $\{ m_N \}$ are uniformly bounded multipliers on $A \subset [-1,1]$, a fundamental domain for $\TT \cong \RR/ \Z $.
If
\[ \left\| \V^r\left( \left( m_N \hat{f} \right)^{\vee} \right) \right\|_{L^2(\RR)} \lesssim \|f\|_{L^2(\RR)}\]
then
\[ \left\| \V^r \left( \left( m_N \hat{f} \right)^{\vee} \right) \right\|_{l^2(\Z)} \lesssim \|f\|_{l^2(\Z)} \]
as well.
\end{lemma}

Since $\widehat{^{0}R_N}$ are supported in (say) $[-1/2,1/2]$, while $\widehat{R_N}$ are supported in $[0,1]$, the above transference lemma applies. In what follows, we therefore (abuse notation and) view our operators as acting on $f \in L^2(\RR)$.

We begin our analysis with the simpler case $\V^r\left( ^{0}R_N*f\right)$ which already contains the core of our method.

\subsection{Proof of Proposition \ref{key}}

\begin{proof}[$\V^r( ^{0}R_N*f)$ is $L^2$-Bounded]
We define, again via the Fourier transform, the more tractable approximate operators
\[ \widehat{ ^{0} T_N }(\alpha) := S_N^0(0/1) \omega_N \chi(N^d m_d \alpha).\]
The claim is that
\[ \sup_\alpha \sum_N \left| \widehat{ ^{0}R_N} - \widehat{^{0}T_N } \right|^2(\alpha) \lesssim 1;\]
assuming this result, a square function argument as above will reduce the problem to bounding $\V^r( ^{0}T_N*f)$ in $L^2$.

Viewing $\alpha$ as fixed, we expand the foregoing as
\[ \aligned
&\sum_N | S_N^0(0/1)|^2 |V_N(\alpha) - \omega_N \chi(N^d m_d \alpha) |^2 |\chi(N^{d-\delta}m_d \alpha)|^2 \\
& \qquad = \sum_j |S_{N_j}^0 (0/1)|^2 \sum_{N: \frac{x_N}{y_N} = \frac{x_{N_j}}{y_{N_j}}}  |V_N(\alpha) - \omega_N \chi(N^d m_d \alpha) |^2 |\chi(N^{d-\delta}m_d \alpha)|^2 \\
& \qquad \lesssim \sum_j |S_{N_j}^0 (0/1)|^2 \sum_{N: \frac{x_N}{y_N} = \frac{x_{N_j}}{y_{N_j}}}  \min\left\{ N^d m_d |\alpha|, \frac{1}{N m_d^{1/d} |\alpha|^{1/d}} \right\}, \endaligned \]
by our twisted ``pseudo-projective'' estimates.
We majorize the previous sum:
\[ \lesssim_\rho \sum_j |S_{N_j}^0 (0/1)|^2 \lesssim \sum_j y_{N_j}^{-2 \nu} \]
for any $\nu < 1/d$, which converges by our sparsity condition on the $\{y_{N_j}\}$. The claim is proved.

We now majorize
\[ \V^r\left(^{0}T_N*f\right) \lesssim \left( \sum_j \left| \V^r \left(^{0}T_N*f : N_j \leq N < N_{j+1}\right) \right|^2 \right)^{1/2} + \V^r \left(^{0} T_{N_j}*f : j  \right). \]
The second sum is majorized $\lesssim \left( \sum_j \left|^{0}T_{N_j}*f\right|^2 \right)^{1/2}$; we use Parseval to estimate its $L^2$-size by
\[ \left( \sum_j y_{N_j}^{-2 \nu} \right)^{1/2} \|f\|_{L^2}  \lesssim \|f\|_{L^2},\]
where we used the trivial bound $|\omega_{N_j} | \leq 1$.

We now estimate the $L^2$ norm of
\[ \V^r \left(^{0}T_N *f : N_j \leq N < N_{j+1} \right) \]
for each $j$.

With
\[ \hat{g_j}(\alpha) := S_{N_j}^0(0/1) \chi(N_j^d m_d \alpha) \hat{f}(\alpha), \]
we have the pointwise majorization
\[ \V^r(^{0}T_N *f : N_j \leq N < N_{j+1}) \lesssim \V^1(\omega_N : N_j \leq N < N_{j+1}) \cdot \V^r\left( \left( \chi(N^d m_d \alpha) \hat{g_j}(\alpha) \right)^{\vee} \right).\]
But $\V^1(\omega_N : N_j \leq N < N_{j+1}) \lesssim_\rho 1$, since if $N<M$ are successive elements of $I$, we have the bounds
\[ \aligned
|\omega_N - \omega_M | &= \left| \int_0^1 e(N t \gamma_{N_j}) - e(Mt \gamma_{N_j}) \ dt \right| \\
&\lesssim \min \left\{ \frac{1}{N |\gamma_{N_j}|}, (M-N) |\gamma_{N_j}| \right\} \\
&\lesssim  \min \left\{ \frac{1}{N |\gamma_{N_j}|}, N |\gamma_{N_j}| \right\}. \endaligned \]
Using a square function argument and Bourgain's \cite[Lemma 3.28]{B1}, we have that
\[ \| \V^r( \left( \chi(N^d m_d \alpha) \hat{g_j}(\alpha) \right)^{\vee} ) \|_{L^2} \lesssim \frac{r}{r-2} \|g_j\|_{L^2} \lesssim_r y_{N_j}^{-\nu} \|f\|_{L^2}.\]
Since this is square-summable in $j$, the result is proved.
\end{proof}

\begin{proof}[$\V^r( R_N*f)$ is $L^2$-Bounded]
We begin by subdividing
\[ \aligned
\widehat{R_N}(\alpha) &:= \sum_{t: 2^t \leq N^\delta} \widehat{R_{N,t}}(\alpha) \\
&:= \sum_{t: 2^t \leq N^\delta} \left( \sum_{j=0}^{m_d -1} \sum_{a/b \in A_{N,t}^j} S_N^j(a/b) V_N\left(\alpha - \frac{1}{m_d}(j+a/b) \right) \chi \left(N^{d -\delta} m_d \left(\alpha - \frac{1}{m_d}(j +a/b)\right)\right) \right) \\
&=
\sum_{t: 2^t \leq N^\delta} \left( \sum_{j=0}^{m_d -1} \sum_{a/b \in A_{N_{l(t)},t}^j} S_N^j(a/b) V_N\left(\alpha - \frac{1}{m_d}(j+a/b) \right) \chi \left(N^{d -\delta} m_d \left(\alpha - \frac{1}{m_d}(j +a/b)\right)\right) \right). \endaligned \]
Our task will be to show that
\[ \sum_t \V^r(R_{N,t}*f : N^{\delta} \geq 2^t )  =
\sum_t \V^r(R_{N,t}*f : N^{\delta} \geq 2^t, N_{l(t)} \leq N < N_{l(t) +1} )
\]
is bounded in $L^2$.
For notational ease, let $N_{m(t)} := \min\{ N \in I : N \geq N_{l(t)}, 2^{t/\delta} \}$; we're interested in showing that
\[ \| \V^r(R_{N,t}*f : N_{m(t)} \leq N \leq N_{l(t)+1} ) \|_{L^2(\RR)} \]
is summable in $t$. (If no scale $N_{l(t)}$ exists, then the pertaining $R_{N,t}$s are just zero operators.)

To this end, we similarly define for $N_{m(t)} \leq N < N_{l(t) + 1}$
\[ \aligned
\widehat{ T_{N,t}} (\alpha) &:= \sum_{j=0}^{m_d -1} \sum_{A_{N,t}^j} S_N^j(a/b) \omega_N
 \chi \left(N^{d} m_d \left(\alpha - \frac{1}{m_d}(j +a/b)\right)\right) \\
&=\sum_{j=0}^{m_d -1} \sum_{A_{N_{l(t)},t}^j} S_N^j(a/b) \omega_N \chi \left(N^{d} m_d \left(\alpha - \frac{1}{m_d}(j +a/b)\right)\right) .  \endaligned \]

We again estimate
\[ \sup_\alpha \sum_{N_{m(t)} \leq N < N_{l(t)+ 1}} |R_{N,t}  - T_{N,t} |^2 (\alpha).\]
To do so, we fix $\alpha$, set
\[ E_N \left( \alpha - \frac{1}{m_d}(j+a/b) \right) := V_N\left(\alpha - \frac{1}{m_d}(j+a/b)\right) - \omega_N \chi \left(N^{d} m_d \left(\alpha - \frac{1}{m_d}(j +a/b)\right)\right)
, \]
where
\[ \left| E_N \left( \alpha - \frac{1}{m_d}(j+a/b)\right) \right| \lesssim \min \left\{ N^d m_d | \alpha - \frac{1}{m_d}(j+a/b)|, \frac{1}{N m_d^{1/d} | \alpha - \frac{1}{m_d}(j+a/b)|^{1/d}} \right\},\]
and consider
\[ \aligned
&|R_{N,t} - T_{N,t}|^2( \alpha) \\
&  = \left| \sum_{j=0}^{m_d -1} \sum_{A_{N_{l(t)},t}^j} S_{N_{l(t)}}^j(a/b) \left( E_N \left( \alpha - \frac{1}{m_d}(j+a/b)\right) \right) \cdot \chi \left(N^{d-\delta} m_d \left(\alpha - \frac{1}{m_d}(j +a/b)\right)\right)  \right|^2
\\
&  = \sum_{j=0}^{m_d -1} \sum_{A_{N_{l(t)},t}^j} | S_{N_{l(t)}}^j(a/b) |^2 \left| E_N \left( \alpha - \frac{1}{m_d}(j+a/b)\right)  \right|^2 \cdot \chi \left(N^{d-\delta} m_d \left(\alpha - \frac{1}{m_d}(j +a/b)\right)\right) ^2, \endaligned \]
using the fact that for $N$ sufficiently large the supports
\[  \chi \left(N^{d-\delta} m_d \left(\alpha - \frac{1}{m_d}(j +a/b)\right)\right) \]
are disjoint in $a/b \in A_N^j$.
Again using the disjoint support assumptions, we may estimate the foregoing by
\[ \max_{a/b \in A_{N_{l(t)},t}} | S_{N_{l(t)}}^j(a/b) |^2 \left| E_N \left( \alpha - \frac{1}{m_d}(j+a/b)\right)  \right|^2 \cdot   \chi \left(N^{d-\delta} m_d \left(\alpha - \frac{1}{m_d}(j +a/b)\right)\right)^2.\]
Now, if there exists no $a/b \in A_{N_{l(t)},t}$ for which
\[  \left|\alpha - \frac{1}{m_d}(j+a/b) \right| \leq \frac{1}{m_d}N_{m(t)}^{\delta - d},\]
then the above is just zero. Otherwise, let
\[ p/q = p(\alpha)/q(\alpha) \in A_{N_{l(t)},t} \]
be the unique such element satisfying the above inequality with pertaining index $k$, i.e.\
\[ \left|\alpha - \frac{1}{m_d}(k+p/q) \right| \leq \frac{1}{m_d}N_{m(t)}^{\delta - d}. \]

In this case we may estimate the above maximum by
\[  \aligned
&| S_{N_{l(t)}}^k(p/q) |^2 \left| E_N\left( \alpha - \frac{1}{m_d}(k+p/q)\right)  \right|^2 \\
& \qquad \lesssim | S_{N_{l(t)}}^k(p/q) |^2 \cdot \min \left\{ N^d m_d | \alpha - \frac{1}{m_d}(k+p/q)|, \frac{1}{N m_d^{1/d} | \alpha - \frac{1}{m_d}(k+p/q)|^{1/d}} \right\}. \endaligned\]

Summing now the above over $N_{m(t)} \leq N < N_{l(t) +1 }$ accrues an upper estimate of
\[ \lesssim_\rho | S_{N_{l(t)}}^k(p/q) |^2 \lesssim 2^{-2t \nu} \]
for $\nu < 1/d$.

Since this is summable in $t$, by another square function argument it suffices to show that
\[ \| \V^r(T_{N,t} *f : N_{m(t)} \leq N < N_{l(t)+1} ) \|_{L^2(\RR)} \]
is summable in $t$.

With
\[ \hat{ g_t} (\alpha) := \sum_{j=0}^{m_d -1} \sum_{a/b \in A_{N_{l(t)},t}^j} S_{N_{l(t)}}^j (a/b)
\chi\left( N_{m(t)}^d m_d \left( \alpha - \frac{1}{m_d}(j+a/b)\right) \right) \hat{f}(\alpha), \]
so
\[ \|g_t\|_{L^2(\RR)} \lesssim 2^{-t\nu} \|f\|_{L^2(\RR)} \]
we majorize
\[ \aligned
&\V^r(T_{N,t} *f : N_{m(t)} \leq N < N_{l(t)+1} ) \\
& \qquad \leq \V^1 (\omega_N : N_{m(t)} \leq N < N_{l(t)+1})  \cdot \V^r( B_{N,t}*g_t: N_{m(t)} \leq N < N_{l(t)+1}), \endaligned\]
where
\[ \widehat{B_{N,t}}(\alpha) := \sum_{j=0}^{m_d -1 } \sum_{a/b \in A_{N_{l(t)},t}^j} \chi\left(N^d m_d \left( \alpha - \frac{1}{m_d}(j+a/b)\right)\right).\]
By arguing as above, we have that
\[ \V^1 (\omega_N : N_{m(t)} \leq N < N_{l(t)+1})   \lesssim_\rho 1.\]
By \cite[Lemma 3.7]{BK}, we may further estimate
\[ \aligned
\| \V^r( B_{N,t}*g_t: N_{m(t)} \leq N < N_{l(t)+1}) \| &\lesssim_\rho \left( \frac{r}{r-2} \log |A_{N_{l(t)},t}| \right)^2 \| g_t\|_{L^2(\RR)} \\
&\lesssim \left( \frac{r}{r-2} \cdot t \right)^2 2^{-t\nu} \|f \|_{L^2(\RR)}. \endaligned\]
Since this is summable in $t$, the proof is complete.
\end{proof}


\begin{thebibliography}{9}
\bibitem{AWW}
I. Assani, Wiener Wintner Ergodic Theorems, World Scientific Publishing Co. Inc., River Edge, NJ,
2003.

\bibitem{ADM}  I. Assani, D. Duncan, R. Moore, \emph{Pointwise characteristic factors for Wiener Wintner double recurrence theorem}, preprint, http://arxiv.org/abs/1402.7094. 

\bibitem{AM} I. Assani, R. Moore, \emph{Extension of Wiener-Wintner double recurrence theorem to polynomials II}, preprint, http://arxiv.org/abs/1409.0463.

\bibitem{AP} I. Assani, K. Presser, \emph{A Survey of the return times theorem}, Ergodic Theory Dyn. Systems, Proceedings of the 2011-2012, UNC- Chapel Hill  workshops, Walter De Gruyter, 2013.


\bibitem{Be}
 S. N. Bernstein, Lecons sur les Propri\'{e}t\'{e}s Extr\^emales et la Meilleure Approximation des Fonctions Analytiques d'une Variable R\'{e}ele, Gauthier-Villars,  Paris, 1926; in: l'Approximation, Chelsea, New York, 1970.



\bibitem{BI}
G. D. Birkhoff. \emph{Proof of the ergodic theorem}, Proc. Natl. Acad. Sci. USA \textbf{17} (1931), 656--660.

\bibitem{bosh94}
M. D. Boshernitzan, \emph{Uniform distribution and Hardy fields}, J. Anal. Math.  \textbf{62}  (1994), 225--240.

\bibitem{bosh/kolesnik/quas/wierdl}
M. Boshernitzan, G. Kolesnik, A. Quas, M. Wierdl,
\emph{Ergodic averaging sequences}, J. Anal. Math.  \textbf{95} (2005), 63--103.

\bibitem{bosh/wierdl}
M. Boshernitzan, M. Wierdl \emph{Ergodic theorems along sequences and Hardy fields}, Proc. Nat. Acad. Sci. U.S.A.  \textbf{93} (1996),  8205--8207.


\bibitem{B0}
J. Bourgain,
\emph{On the maximal ergodic theorem for certain subsets of the integers},
Israel J. Math. \textbf{61} (1988), no. 1, 39--72.


\bibitem{B1}
J. Bourgain,
Pointwise ergodic theorems for arithmetic sets. Inst. Hautes \'{E}tudes Sci. Publ. Math.  \textbf{69} (1989), 5--45. With an appendix by the author, Harry Furstenberg, Yitzhak Katznelson and Donald S. Ornstein.

\bibitem{BWW}
J. Bourgain, \emph{Double recurrence and almost sure convergence}, J. Reine Angew. Math. \textbf{404} (1990),
140--161.



\bibitem{C}
A. Calder\'{o}n, \emph{Ergodic theory and translation invariant operators}, Proc. Nat. Acad. Sci., USA \textbf{59} (1968), 349--353.

\bibitem{EFHN} T. Eisner, B. Farkas, M. Haase, R. Nagel, Operator Theoretic Aspects of Ergodic Theory, Graduate Texts in Mathematics, Springer, to appear.


\bibitem{ET}
T. Eisner, T. Tao, \emph{Large values of the Gowers-Host-Kra seminorms}, J. Anal. Math. \textbf{117} (2012), 133--186.

\bibitem{EZ}
T. Eisner, P. Zorin-Kranich, \emph{Uniformity in the Wiener-Wintner theorem for nilsequences}, Discrete Contin. Dyn. Syst.  \textbf{33} (2013), 3497--3516.

\bibitem{falconer} K. Falconer, Fractal Geometry: Mathematical Foundations and Applications, 2nd ed., John Wiley \& Sons Inc., Hoboken, NJ, 2003.

\bibitem{frantWW} N. Frantzikinakis, \emph{Uniformity in the polynomial Wiener-Wintner theorem},   Ergodic Theory Dynam. Systems  \textbf{26}  (2006), 1061--1071.

\bibitem{frantz09} N. Frantzikinakis, \emph{Equidistribution of sparse sequences on nilmanifolds} J. Anal. Math.  \textbf{117} (2012), 133--186.

\bibitem{EZ}
T. Eisner, P. Zorin-Kranich, \emph{Uniformity in the Wiener-Wintner theorem for nilsequences}, Discrete Contin. Dyn. Syst.  \textbf{33} (2013), 3497--3516.

\bibitem{falconer} K. Falconer, Fractal Geometry: Mathematical Foundations and Applications, 2nd ed., John Wiley \& Sons Inc., Hoboken, NJ, 2003.

\bibitem{frantWW} N. Frantzikinakis, \emph{Uniformity in the polynomial Wiener-Wintner theorem},   Ergodic Theory Dynam. Systems  \textbf{26}  (2006), 1061--1071.

\bibitem{frantz09} N. Frantzikinakis, \emph{Equidistribution of sparse sequences on nilmanifolds} J. Anal. Math.   \textbf{109}  (2009), 353--395.

\bibitem{frantz10} N. Frantzikinakis, \emph{Multiple recurrence and convergence for Hardy sequences of polynomial growth}, J. Anal. Math.   \textbf{112}  (2010), 79--135.

\bibitem{frantz/johnson/lesigne/wierdl} N. Frantzikinakis, M. Johnson, E. Lesigne, M. Wierdl, \emph{Powers of sequences and convergence of ergodic averages}, Ergodic Theory Dynam. Systems  30  (2010), 1431--1456.

\bibitem{frantz/wierdl} N. Frantzikinakis, M. Wierdl, \emph{A Hardy field extension of Szemer\'edi's theorem}, Adv. Math.   \textbf{222}  (2009), 1--43.

\bibitem{HK}
B. Host, B. Kra, \emph{Uniformity seminorms on $l^\infty$ and applications}, J. Anal. Math.  \textbf{108} (2009), 219--276.

\bibitem{HK05}
B. Host, B. Kra, \emph{Nonconventional ergodic averages and nilmanifolds}, Ann. Math. (2)  \textbf{161} (2005), no. 1, 397--488. 

\bibitem{HUA}
L-K. Hua.
Introduction to Number Theory.
Second edition. Springer-Verlag, Berlin, 1982.


\bibitem{hutchinson} J.-E. Hutchinson, \emph{Fractals and self-similarity}, Indiana Univ. Math. J.   \textbf{30}  (1981), 713--747.

\bibitem{jarnik}
V. Jarnik, \emph{Zur metrischen Theorie der diophantischen Approximationen}, Prace Mat.-Fiz.  \textbf{36} (1928/9), 91--106.

\bibitem{jenkinson}
O. Jenkinson, \emph{On the density of Hausdorff dimensions of bounded type continued fraction sets: the Texan conjecture}, Stochastics Dyn.  \textbf{4} (2004), 63--76.

\bibitem{texan-conj}
M. Kesseb\"{o}hmer, S. Zhu, \emph{Dimension sets for infinite IFSs: the Texan conjecture}, J. Number Theory   \textbf{116}  (2006),  230--246.

\bibitem{khintchine} A. Ya. Khintchine, Continued fractions, English transl. by P. Wynn, Noordhoff, Groningen, 1963.



\bibitem{BK}
B. Krause, \emph{Polynomial Ergodic Averages Converge Rapidly: Variations on a Theorem of Bourgain}. Preprint, http://arxiv.org/pdf/1402.1803v1.pdf.

\bibitem{KZ}
B. Krause, P. Zorin-Kranich, \emph{A random pointwise ergodic theorem with Hardy field weights}.  Preprint, http://arxiv.org/pdf/1410.0806v1.pdf.


\bibitem{kuipers/niederreiter}  L. Kuipers, H. Niederreiter, Uniform Distribution of Sequences. Pure and Applied Mathematics, Wiley-Interscience, New York-London-Sydney, 1974.

\bibitem{lesigneWW} E. Lesigne, \emph{Spectre quasi-discret et th\'eor\`eme ergodique de Wiener-Wintner pour les polyn\^omes}, Ergodic Theory Dynam. Systems  \textbf{13}  (1993), 767--784.

\bibitem{MT}
M. Mirek; B. Trojan, \emph{Discrete maximal functions in higher dimensions and applications to ergodic theory}. Preprint, http://arxiv.org/abs/1405.5566.


\bibitem{montgomery} H. L. Montgomery,  \emph{Harmonic Analysis as Found in Analytic Number Theory}, in Twentieth Century Harmonic Analysis--A Celebration. Proceedings of the NATO Advanced Study Institute Held in Il Ciocco, July 2-15, 2000 (Ed. J. S. Byrnes). Dordrecht, Netherlands: Kluwer, pp. 271--293, 2001.


\bibitem{RW}
J. Rosenblatt; M. Wierdl,  \emph{Pointwise ergodic theorems via harmonic
analysis}. Ergodic theory and its connections with harmonic analysis
(Alexandria, 1993). London Math. Soc. Lecture Note Ser., 205, Cambridge
Univ. Press, Cambridge, (1995), 3--151.


\bibitem{S}
E. M. Stein, Harmonic analysis: real-variable methods, orthogonality, and oscillatory integrals. Princeton Mathematical Series, 43. Monographs in Harmonic Analysis, III. Princeton University Press, Princeton, NJ, 1993.

\bibitem{tao-epsilon} T. Tao, An Epsilon of Room, Real Analysis: Pages from Year Three of a Mathematical Blog. Graduate Studies in Mathematics 117, American Mathematical Society, 2010.


\bibitem{T}
J-P. Thouvenot,
\emph{La convergence presque s\^{u}re des moyennes ergodiques suivant certaines sous-suites d'entiers (d'apr\`{e}s Jean Bourgain)},
S\'{e}minaire Bourbaki, Vol. 1989/90.
Ast\'{e}risque No. 189-190 (1990), Exp. No. 719, 133--153.


\bibitem{VA}
Vaughan, R. C. The Hardy-Littlewood method. Second edition. Cambridge Tracts in Mathematics, 125. Cambridge University Press, Cambridge, 1997.

\bibitem{WW}
N. Wiener; A. Wintner, \emph{Harmonic analysis and ergodic theory}, Amer. J. Math. \textbf{63} (1941), 415--426.































































\end{thebibliography}
\end{document}